\documentclass[11pt]{article}
\usepackage[table,xcdraw]{xcolor}
\usepackage{adjustbox}
\usepackage{geometry}
\usepackage[english]{babel}
\usepackage{amsmath}
\usepackage{amsthm}
\usepackage{natbib}
\usepackage[affil-it]{authblk}
\usepackage{graphicx}
\usepackage{comment}
\usepackage{amsfonts}
\usepackage{todonotes}

\newtheorem{theorem}{Theorem}
\newtheorem{definition}{Definition}
\usepackage{mathrsfs}
\usepackage{float}
\usepackage{booktabs}
\usepackage{multirow}
\usepackage{complexity}
\usepackage{subfig}
\usepackage{caption}
\usepackage{upgreek}
\usepackage{hyperref}
\hypersetup{
  colorlinks   = true,
  urlcolor     = cyan,
  linkcolor    = blue,
  citecolor   = red
}
\usepackage{paralist}

\newcommand{\revised}[1]{{\color{black}{#1}}}

\makeatletter
\def\blfootnote{\xdef\@thefnmark{}\@footnotetext}
\makeatother

\makeatletter
\def\ps@pprintTitle{%
  \let\@oddhead\@empty
  \let\@evenhead\@empty
  \let\@oddfoot\@empty
  \let\@evenfoot\@oddfoot
}
\makeatother
\title{Mathematical Programming Formulations for the Collapsed k-Core Problem\blfootnote{\textcopyright 2023. This manuscript version is made available under the CC-BY-NC-ND 4.0 license \url{https://creativecommons.org/licenses/by-nc-nd/4.0/}. Accepted for publication in European Journal of Operational Research; doi: \textcolor{cyan}{\textbf{10.1016/j.ejor.2023.04.038}}\\ \hspace*{0.6cm}\textit{Email addresses:} \texttt{cerulli@essec.edu} (M. Cerulli), \texttt{dserra@unisa.it} (D. Serra), \texttt{csorgente@unisa.it} (C. Sorgente), \texttt{archetti@essec.edu} (C. Archetti), \texttt{ivana.ljubic@essec.edu} (I. Ljubi\'c)}}
\date{}

\author[1]{Martina Cerulli}
\author[2]{Domenico Serra}
\author[2]{Carmine Sorgente}
\author[1]{Claudia Archetti}
\author[1]{Ivana Ljubi\'c}
	
	\affil[1]{ESSEC Business School, Cergy-Pontoise, 95021, France}
	
	\affil[2]{University of Salerno, Fisciano, 84084, Italy}

\begin{document}
\maketitle
\begin{abstract}
In social network analysis, the size of the $k$-core, i.e., the maximal induced subgraph of the network with minimum degree at least $k$, is frequently adopted as a typical metric to evaluate the cohesiveness of a community.
  We address the Collapsed $k$-Core Problem, which seeks to find a subset of $b$ users, namely the most critical users of the network, the removal of which results in the smallest possible $k$-core.
  For the first time, both the problem of finding the $k$-core of a network and the Collapsed $k$-Core Problem are formulated using mathematical programming.
  On the one hand, we model the Collapsed $k$-Core Problem as a natural deletion-round-indexed Integer Linear formulation. On the other hand, we provide two bilevel programs for the problem, which differ in the way in which the $k$-core identification problem is formulated at the lower level. The first bilevel formulation is reformulated as a single-level sparse model, exploiting a Benders-like decomposition approach.
 To derive the second bilevel model, we provide a linear formulation for finding the $k$-core and use it to state the lower-level problem. We then dualize the lower level and obtain a compact Mixed-Integer Nonlinear single-level problem reformulation.
  We additionally derive a combinatorial lower bound on the value of the optimal solution and describe some pre-processing procedures, and valid inequalities for the three formulations.
  The performance of the proposed formulations is compared on a set of benchmarking instances with the existing state-of-the-art solver for mixed-integer bilevel problems proposed in \citep{ljubic2017}.

\end{abstract}

\section{Introduction}
In recent years, with the advent of social networks era, studying the behaviour of the users in a network has gained increasing interest. Particularly important are the so-called critical users, i.e., the ones who have a large number of connections with other users and whose \revised{departure} from the network might potentially cause the exit of many other users. Indeed, a property of social networks is that the decision of each user (leaving or remaining in the network) is influenced by that of her connected friends. A popular assumption is that a user remains in the network if she has at least a certain number of connections, say $k$, in the network \citep{bhawalkar2015}. On the contrary, a user is driven to leave the social network if she has less than $k$ connections. Given this assumption, if a user leaves the network, the degree of her neighbors decreases by one, eventually becoming smaller than $k$ for some of them.
Thus, a cascade phenomenon is observed each time a user drops out from the network, until a \textit{stable} configuration is obtained, which corresponds to the $k$-core of the social network graph. 

In this context, we study the Collapsed $k$-Core Problem, which has been introduced in \citep{zhang2017} to identify the critical users to be eventually incentivized not to leave the network (or, from an adversarial point of view, to leave it). This problem indeed consists in finding the set of a given number $b$ of users, whose exit from the network minimizes the number of the remaining users in the network itself, i.e., leads to the minimal $k$-core.
First of all, we propose a formulation of the problem modeling the cascade effect which determines the withdrawal of a certain number of users from the network. In such approach, a time index is needed to represent the subsequent deletion rounds of this process.
Beyond that, the tools of bilevel optimization have been recently used for developing exact methods for several critical node/edge detection problems \citep{Furini-et-al2019,Furini-et-al2020,Furini-et-al2020b,Furini-et-al2022}. These works show that novel and computationally effective mathematical programming formulations can be derived thanks to the bilevel interpretation of the problems. Motivated by these studies, we use bilevel programming to model the Collapsed $k$-Core Problem, discarding the time index. A bilevel program is an optimization problem where one problem is nested into another \citep{vicente1994,colson,dempe2002,cerul2021,ivana_survey}. The formulation of a classical bilevel problem reads 
\begin{align*}\label{eq:bilP}
  \min\limits_{x \in \mathcal{X}, y} &\: F(x,y) \\
  \text{s.t.} &\; G(x,y) \geq 0 \tag{\mbox{$\mathsf{P}$}}\\
  & y \in \arg\min\limits_{y' \in \mathcal{Y}} \{ f(x,y')~|~g(x,y') \geq 0\}
\end{align*}
The outer optimization problem in the variables $x$ and $y$ is the so-called upper-level problem. The inner optimization problem in the variable $y'$, parameterized with respect to the upper-level variables $x$, is the so-called lower-level problem. In formulation \eqref{eq:bilP}, we implicitly assume that the lower-level problem has only one optimal solution for each value of $x$. If this is not the case, this formulation is the one obtained with the so-called \textit{optimistic} approach \citep{dempe2002}, which consists in selecting the lower-level optimal solution corresponding to the best outcome for the upper level that minimizes it. The whole bilevel problem can be seen as a hierarchical decision process: in the upper level, a \emph{leader} makes a decision while anticipating the optimal reaction of the lower-level decision maker, the \emph{follower}, whose decision depends on the decision of the leader. Another way to formulate problem $(\mathsf{P})$ is
\begin{align*}
  \min\limits_{x \in \mathcal{X}, y} &\; F(x,y) \\
  \text{s.t.} &\; G(x,y) \geq 0, \quad g(x,y) \geq 0 \\
  &\; f(x,y) \leq \varphi(x),
\end{align*}
where $\varphi(x) = \min\limits_{y' \in \mathcal{Y}} \{ f(x,y')~|~g(x,y') \geq 0\}$ is the so called \textit{value function} of the lower level.

In the Collapsed $k$-Core Problem, the hierarchical structure can be described as follows. We can see the follower as an entity who is computing the collapsed $k$-core resulting after the decision of the leader on the $b$ nodes to interdict from the network. So, the follower aims at identifying the subgraph of maximum cardinality, resulting from the interdiction of the $b$ nodes selected by the leader, satisfying the property that each node of the subgraph has at least $k$ neighbors. The leader instead aims at detecting the set of $b$ nodes for which the cardinality of the associated subgraph is minimized, which corresponds to the set of the most critical users in the network. 

\paragraph{Contributions}
It is well-known that the problem of finding the $k$-core of a graph, denoted as $k$-Core Detection Problem in the following, can be solved in polynomial time (see \cite{batagelj2003}), and one can easily model the problem using binary variables. However, to the best of our knowledge, prior to the current work, no {Integer Linear Programming (ILP)} or Linear Programming (LP) formulation (and, thus, a polynomial approach based on LP) for the $k$-Core Detection Problem was known. In this work we provide a first compact LP formulation for calculating the $k$-core. In addition, as far as we know, no mathematical programming formulations for the Collapsed $k$-Core Problem \citep{zhang2017} have ever been investigated in the existing literature. We provide three integer programming formulations, and propose to enhance them with a combinatorial lower bound and valid inequalities. The first formulation is a compact time-expanded model that mimics the iterative node ``collapsing'' process. The two other models are based on the bilevel interpretation of the problem. The first is a sparse formulation that projects out the lower-level variables and exploits a Benders-like decomposition. The second exploits the LP-based formulation of the $k$-Core Detection Problem, and the LP-duality. All the three formulations are implemented and computationally evaluated against the state-of-the-art bilevel solver from \citet{ljubic2017}.
The obtained results demonstrate the efficiency of the proposed approaches, with the model using the LP-duality exhibiting the best performances, both in terms of computing times and gaps at the termination.

The rest of the paper is organized as follows. In Section~\ref{sec:literature}, we review the main literature in the field. In Section~\ref{sec:notation} we introduce the needed definitions and notations. In Section~\ref{sec:kcore}, we present two mathematical programming formulations of the $k$-Core Detection Problem. In Section~\ref{sec:math-formulations}, we introduce different mathematical programming formulations for the Collapsed $k$-Core Problem: first a time-dependent model, then two bilevel programs, which mainly differ in the lower level, where the two formulations proposed in Section~\ref{sec:kcore} are used. In Section~\ref{sec:valid_ineq}, we describe pre-processing procedures, as well as valid inequalities that strengthen the proposed formulations. In Section~\ref{sec:separation-procedures}, we describe how to separate the valid inequalities which are exponential in number. Section~\ref{sec:numerical-experiments} is devoted to the numerical experiments, and Section~\ref{sec:conclusion} concludes the paper.

\vspace*{-2mm}
\section{Literature Review}\label{sec:literature}
We consider the following definition of the $k$-core:
\begin{definition}\label{def:kcore}
Given an undirected graph $\mathcal{G}=(V,E)$, and a positive integer $k$, the $k$-core of $\mathcal{G}$ is the maximal 
induced subgraph of $\mathcal{G}$ in which all the nodes have degree at least $k$. \end{definition} 
The $k$-core may be calculated as the resulting graph obtained by iteratively deleting from $\mathcal{G}$ all the nodes that have degree less than $k$, in any order. This procedure is known as the $k$-core decomposition \citep{sariyuce2016}, the basic theory of which is surveyed in~\citep{malliaros2020}. 
We point out that the original definition of the $k$-core given by \citet{Seidman1983NetworkSA} requires the connectedness of the $k$-core. In \citep{matula1983}, this original definition is used, and an algorithm, named Level Component Priority Search, is introduced for finding all the maximal connected components of a graph with degree at least $k$. However, in most of the recent papers in the field, the connectedness requirement is relaxed, starting from \citet{batagelj2003}, where an algorithm with time complexity $O(|E|)$ for core decomposition is proposed. This is why in Definition~\ref{def:kcore} we also assume that the $k$-core may contain multiple connected components.

The \textit{Anchored $k$-Core Problem} is studied by \citet{bhawalkar2015}. It consists in anchoring $b$ nodes (nodes that remain engaged no matter what their friends do) to maximize the size of the corresponding anchored $k$-core, i.e., the maximal subgraph in which every non-anchored node has degree at least $k$. It may be useful to incentivize key individuals to stay engaged within the network, preventing the cascade effect. The Anchored $k$-Core Problem is solvable in polynomial time for $k\leq 2$, but is \NP-hard for $k > 2$ \citep{bhawalkar2015}. In \citep{zhang2017OLAK} a greedy algorithm, called
OLAK, is proposed to solve this problem, while in \citep{laishram2020residual} the so-called Residual Core Maximization heuristic is introduced.
Another paper focusing on the maximization of the $k$-core is \citep{chitnis2013}, where improved hardness results on the Anchored $k$-Core Problem are given. 

From an antagonistic perspective, a natural question associated with the Anchored $k$-Core Problem is how to maximally collapse the engagement of the network by incentivizing the $b$ most critical users to leave. This problem is called the \textit{Collapsed $k$-Core Problem}, and was introduced in \citep{zhang2017}. The aim is to find the set of $b$ nodes the deletion of which leads to the smallest $k$-core (i.e., the $k$-core with minimum cardinality), obtained from $\mathcal{G}$ by iteratively removing nodes with degree strictly lower than $k$.
In \citep{zhang2017}, it is shown that the Collapsed $k$-Core Problem is \NP-hard for any $k\geq 1,$ and a greedy algorithm to compute feasible solutions for the problem is proposed. In \citep{luo2021parameterized}, the parameterized complexity of the Collapsed $k$-Core Problem is studied with respect to the parameters $b$, $k$, as well as another parameter, say $\gamma$, which represents the maximum allowed cardinality of the remaining $k$-core (i.e., it should not have more than $\gamma$ number of nodes).
A further study on the $k$-core minimization was conducted in \citep{zhu2018}, where the focus was on edges instead of nodes. Indeed, the aim of the problem in this case is to identify a set of $b$ edges, so that the minimal $k$-core is obtained by deleting these edges from the graph. This problem is proven to be \NP-hard, and a baseline greedy algorithm is proposed.

In \citep{shuo2019}, the Collapsed $k$-Core Problem is applied to determine the key scholars who should be engaged in various tasks in science and technology management, such as experts finding, collaborator recommendation, and so on. More recently, in \citep{zhang2020}, a related problem is studied: the Collapsed $(\alpha, k)$-\NP-community in signed graph. Given a signed graph, which is a graph where each edge has either a ``$+$'' sign if representing a friendship or a ``$-$'' sign if representing an enemy relationship, the $(\alpha, k)$-\NP-community is made of the users having not less than $\alpha$ friends and no more than $k$ enemies in the community. This problem aims at finding the set of nodes the removal of which from the graph will lead to the minimum cardinality of the remaining $(\alpha, k)$-\NP-community.

\section{Notations}\label{sec:notation}
Let us consider an undirected graph $\mathcal{G}=(V,E)$, with $n=|V|$, and two positive integers $k$ and $b$. Denote by $N(i)$ the set of neighbors of node $i$ in graph $\mathcal{G},$ i.e., the set $\{j \in V : \{i,j\} \in E \}$. Furthermore, let $\delta(\mathcal{G})$ be the minimum degree of the graph $\mathcal{G}$, which is the minimum of its nodes' degrees. For a given set of nodes $S \subseteq V$, let $\mathcal{G}[S]$ denote the subgraph of $\mathcal{G}$ induced by $S$, and thus $\delta(\mathcal{G}[S])$ its minimum degree. Moreover, let $\mathcal{G}\backslash S$ denote the subgraph of $\mathcal{G}$ induced by $V\backslash S$, i.e., $\mathcal{G}[V\backslash S]$. 
Let $C_k(\mathcal{G})$ be the $k$-core of graph $\mathcal{G}$, i.e., the maximum cardinality subgraph in which every node has a degree at least $k$. In the following, we will refer to the number of nodes in a given $k$-core as its \textit{size} or \textit{cardinality}.
\revised{Let us further define the \textit{$k$-subcores} of $\mathcal{G}$ as the induced subgraphs of $\mathcal{G}$ in which all the nodes have degree at least $k$. From this definition, it is clear that the $k$-core of $\mathcal{G}$ coincides with its $k$-subcore of largest cardinality and, while there may exist several $k$-subcores,  the $k$-core of $\mathcal{G}$ is always unique.}
Finally, we say that a node $v$ of $\mathcal{G}$ has coreness (or core number) $k$ if it belongs to a $k$-subcore, but not to any $(k+1)$-subcore.
\vspace*{-0.2cm}
\begin{figure}[H]
  \centering
  \includegraphics[width=.4\textwidth]{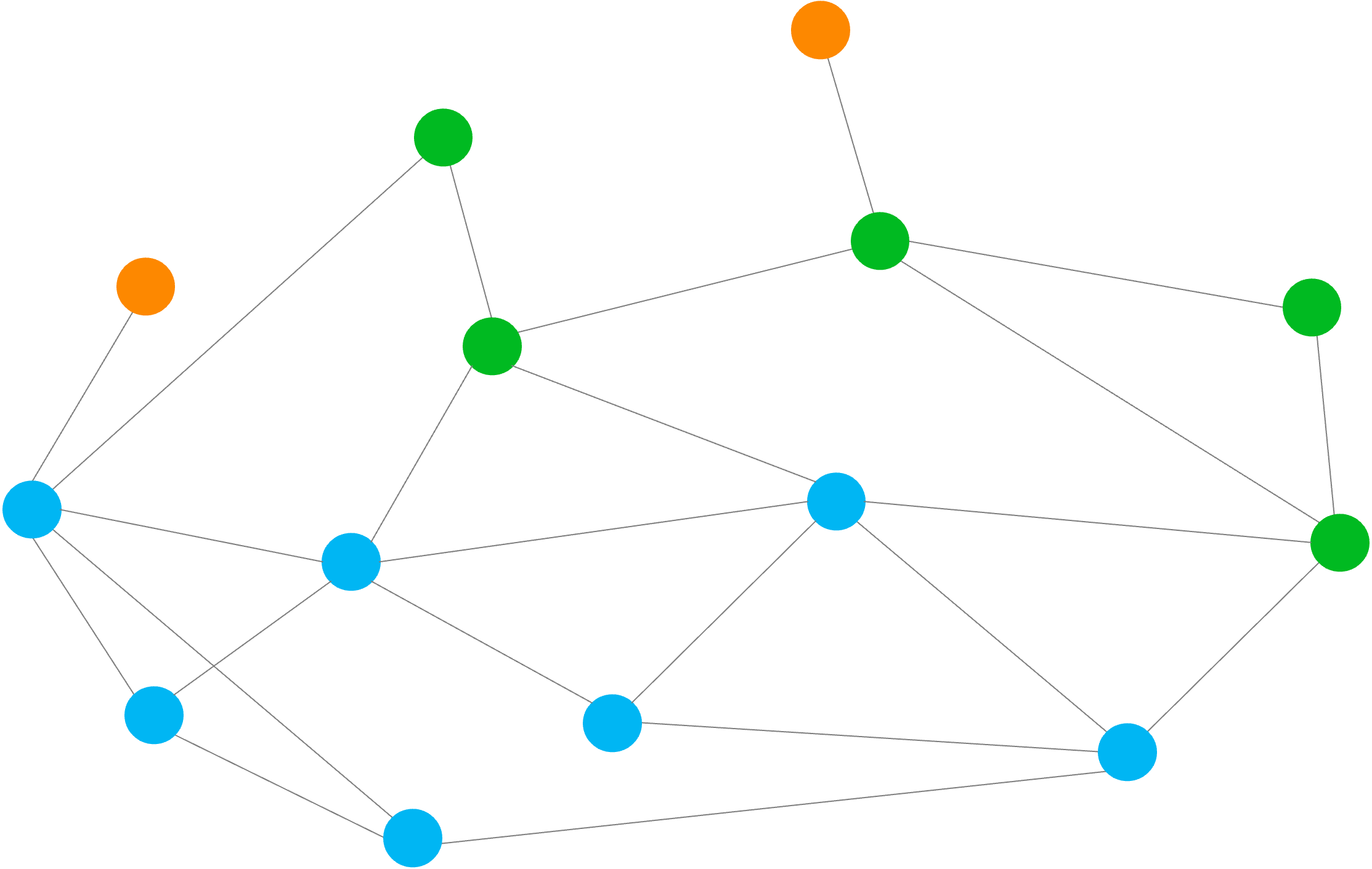}\vspace*{-2mm}
  \caption{\textbf{Example of coreness distribution.}}
  \label{fig:example_coreness}
\end{figure}

In Figure~\ref{fig:example_coreness}, a graph is shown and each node is colored according to its core number. In particular, the two orange nodes have coreness $1$, meaning that they do not belong to any $k$-subcore with $k \ge 2$; the five green nodes, instead, have coreness two, since $2$ is the maximum value of $k$ such that there exist a $k$-subcore including any of them; finally, the remaining seven cyan nodes have coreness $3$, indeed they together form a $3$-subcore.

\revised{A detailed list of the notations used in the paper can be found in~\ref{appendix1}.}
\section{Mathematical formulations for the $k$-Core Detection Problem}\label{sec:kcore}
In this section we propose two formulations for the $k$-Core Detection Problem. The first one is an ILP model defined in the space of binary variables associated with the set of nodes. The second one is a compact, extended formulation which models the subgraph induced by the set of nodes that are outside the $k$-core. In this second model, which uses node and edge variables, the integrality constraints can be relaxed, and hence LP formulation can be obtained.

\subsection{ILP Formulation}\label{subsec:kcore1}
The binary variables involved in the first formulation are defined as:
\begin{equation}\label{eq:kcore1-var}
		y_i = \begin{cases}
			1 &\text{if node $i$ belongs to the $k$-core}\\
			0 &\text{otherwise}
		\end{cases} \quad i \in V.
\end{equation}
Let
\begin{equation}\label{eq:kcore1Y}
		\mathcal{Y} = \left\{ y \in \{0,1\}^n : k \revised{\cdot} y_i \le \sum_{j \in N(i)} y_j,\, \forall\,i \in V \right\} 
\end{equation} 
be the set of incident vectors of any subgraph in $\mathcal{G}$ the nodes of which have degree at least $k$ \revised{(i.e., any $k$-subcore of $\mathcal{G}$)}. With each $ \tilde y \in \mathcal{Y}$ we can associate a subset of nodes $K = \{i \in V: \tilde y_i = 1\}$, and with the whole set $\mathcal{Y}$ of $y$ vectors, we can associate the set $\mathcal{K}\, \revised{ = \{K \subseteq V ~:~ \delta(\mathcal{G}[K]) \geq k\}}$.
For a given $K\in \mathcal{K}$, any node $i\in K$ has a degree at least $k$ in the induced subgraph $\mathcal{G}[K]$, i.e., $\delta(\mathcal{G}[K])\ge k$. Note that, in this case, by definition, $\mathcal{G}[K]$ is the $k$-core of $\mathcal{G}[K]$ itself. In Section~\ref{sec:notation}, we defined such subgraphs $\mathcal{G}[K]$, for each $K \in \mathcal{K}$, as \textit{$k$-subcores} of $\mathcal{G}.$ 

The $k$-Core Detection Problem corresponds to finding the $k$-core of graph $\mathcal{G}$, and can be modeled as the following integer program: 
\begin{equation}\label{eq:kcore1}
  \max_{y \in \mathcal{Y}}\sum_{i \in V} y_i.
\end{equation}
Note that constraints $k \revised{\cdot} y_i \le \sum_{j \in N(i)} y_j$ (in Eq.~\eqref{eq:kcore1Y}) ensure, when maximizing $\sum_{i \in V} y_i$, that no node with induced degree lower than $k$ is selected in the $k$-core.

\begin{figure}[H]
  \centering
  \includegraphics[width=.4\textwidth]{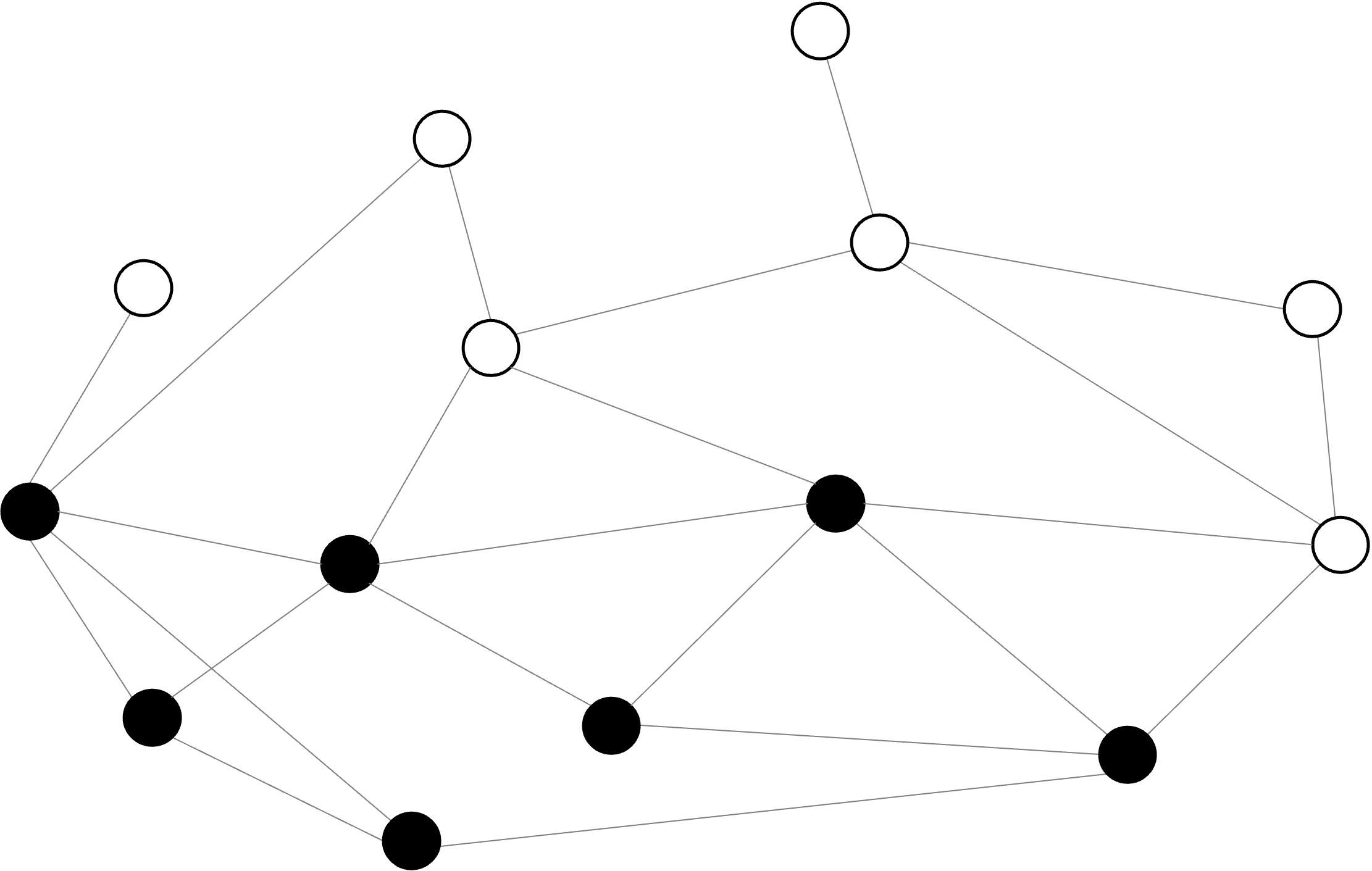}
  \caption{Nodes in the $3$-core of graph in Figure~\ref{fig:example_coreness}.}
  \label{fig:kcore1}
\end{figure}

In Figure~\ref{fig:kcore1}, using the graph in Figure~\ref{fig:example_coreness} with $k=3$, we report in black the nodes for which the variables $y_i$ of formulation~\eqref{eq:kcore1} are $1$, i.e., the nodes belonging to the $3$-core of the graph.

\subsection{LP Formulation}\label{subsec:kcore2}
In this section, we propose an alternative formulation of the $k$-Core Detection Problem by
focusing on the subgraph induced by the subset of nodes which are outside the $k$-core.
Let us define the variables $u_i$ which identify the nodes \textbf{outside} of the $k$-core of the graph:
\begin{equation}\label{eq:kcore2-var}
	u_i = \begin{cases}
		1 &\text{if node $i$ does not belong to the $k$-core}\\
		0 &\text{otherwise}\end{cases} \quad i \in V.
\end{equation}
Note that that variable $u_i$ corresponds to $1-y_i$ in formulation~\eqref{eq:kcore1}. 
Expressed in the space of $u$ variables, \revised{an alternative formulation for the $k$-Core Detection Problem with respect to formulation~\eqref{eq:kcore1} reads:} 

\begin{subequations}\label{eq:kcore2}
\begin{align}
  \max_{u} &\; n - \sum_{i \in V} u_i & \label{eq:kcore2-obj}\\
  \text{s.t.} & \; \sum_{j \in N(i)} u_j + k - |N(i)| \leq k\revised{\cdot}u_i &\forall\,i \in V \label{eq:kcore2-constr}\\ 
  & \; u \in \{0,1\}^{n} & 
\end{align}
\end{subequations}

Indeed, we want to find the $k$-subcore of maximum cardinality, which means that, in the objective function~\eqref{eq:kcore2-obj}, we maximize the difference between $n$ and the sum of $u_i$, that is the number of vertices which are outside the $k$-core. Constraints~\eqref{eq:kcore2-constr} imply that $u_i$ is $1$ (i.e., node $i$ is in not in the $k$-core) iff the difference between its degree and the number of neighbors not in the $k$-core (where this difference corresponds to the number of its neighbors in the $k$-core) is less than $k$. 

Following the idea proposed in \citep[Lemma~1]{gillen}, we can modify the right-hand side of constraints~\eqref{eq:kcore2-constr} as:
\begin{equation}\label{eq:kcore2-constrnew}
  \sum_{j \in N(i)} u_j + k - |N(i)| \leq \Big(\sum_{j \in N(i)} u_j + k -|N(i)| \Big) u_i.
\end{equation}
If $|N(i)| - \sum_{j \in N(i)} u_j < k$ (i.e., the difference between degree of node $i$ and the number of its neighbors not in the $k$-core is less than $k$), $u_i$ will be set to 1. Otherwise, if $|N(i)| - \sum_{j \in N(i)} u_j \geq k$, $u_i$ can be set either to $0$ or $1$, but, since we are minimizing the sum of all $u_i$, it will be set to $0$. 
The resulting formulation is bilinear, and, thus, possibly difficult to solve, due to the presence of the bilinear terms $u_iu_j$. 
We therefore linearize it through the McCormick technique, i.e., by introducing additional binary variables $x_{ij}=u_iu_j$ associated with the edges of $\mathcal{G}$.
Problem~\eqref{eq:kcore2} can be thus reformulated as:

\vspace*{-1mm}
{\small{\begin{subequations}\label{eq:kcore2-mccormick}
\begin{align}
  \max_{x,u} &\; n - \sum_{i \in V} u_i & \\
  \text{s.t.} & \sum_{j \in N(i)} u_j + k - |N(i)| \leq \sum\limits_{j \in N(i)}x_{ij} + (k - |N(i)|)u_i &\forall\,i \in V \label{eq:kcore2-mccormick0}\\ 
  &\; x_{ij} \leq u_i & \forall\,i \in V, j \in N(i) \label{eq:kcore2-mccormick1}\\
  &\; x_{ij} \leq u_j & \forall\,i \in V, j \in N(i)\label{eq:kcore2-mccormick2}\\
  &\; x_{ij} - u_i - u_j \geq -1 & \forall\,i \in V, j \in N(i)\label{eq:kcore2-mccormick3}\\
  &\; u \in \{0,1\}^n, \; x\in \{0,1\}^{|E|} \label{eq:kcore2-sets}
\end{align}
\end{subequations}}}
The $x$ variables represent the edges of the subgraph induced by the nodes outside of the $k$-core.
Indeed, due to constraints~\eqref{eq:kcore2-mccormick1}--\eqref{eq:kcore2-mccormick3}, in any feasible solution of the model, two nodes $i$ and $j$ are connected by an edge (i.e., $x_{ij}=1$) if and only if $u_i=u_j=1$. 
Moreover, if $u_i=0$ (i.e., the node $i$ is in the $k$-core) the right-hand-side of constraints \eqref{eq:kcore2-mccormick0} becomes zero because of constraints~\eqref{eq:kcore2-mccormick1}, implying that $|N(i)| - \sum_{j \in N(i)} u_j \ge k$, which is equivalent to $ \sum_{j \in N(i)} (1-u_j) \ge k$, guaranteeing that the number of neighbors of any node $i$ in the $k$-core is at least $k$.
\begin{figure}[h!]
  \centering
  \includegraphics[width=.4\textwidth]{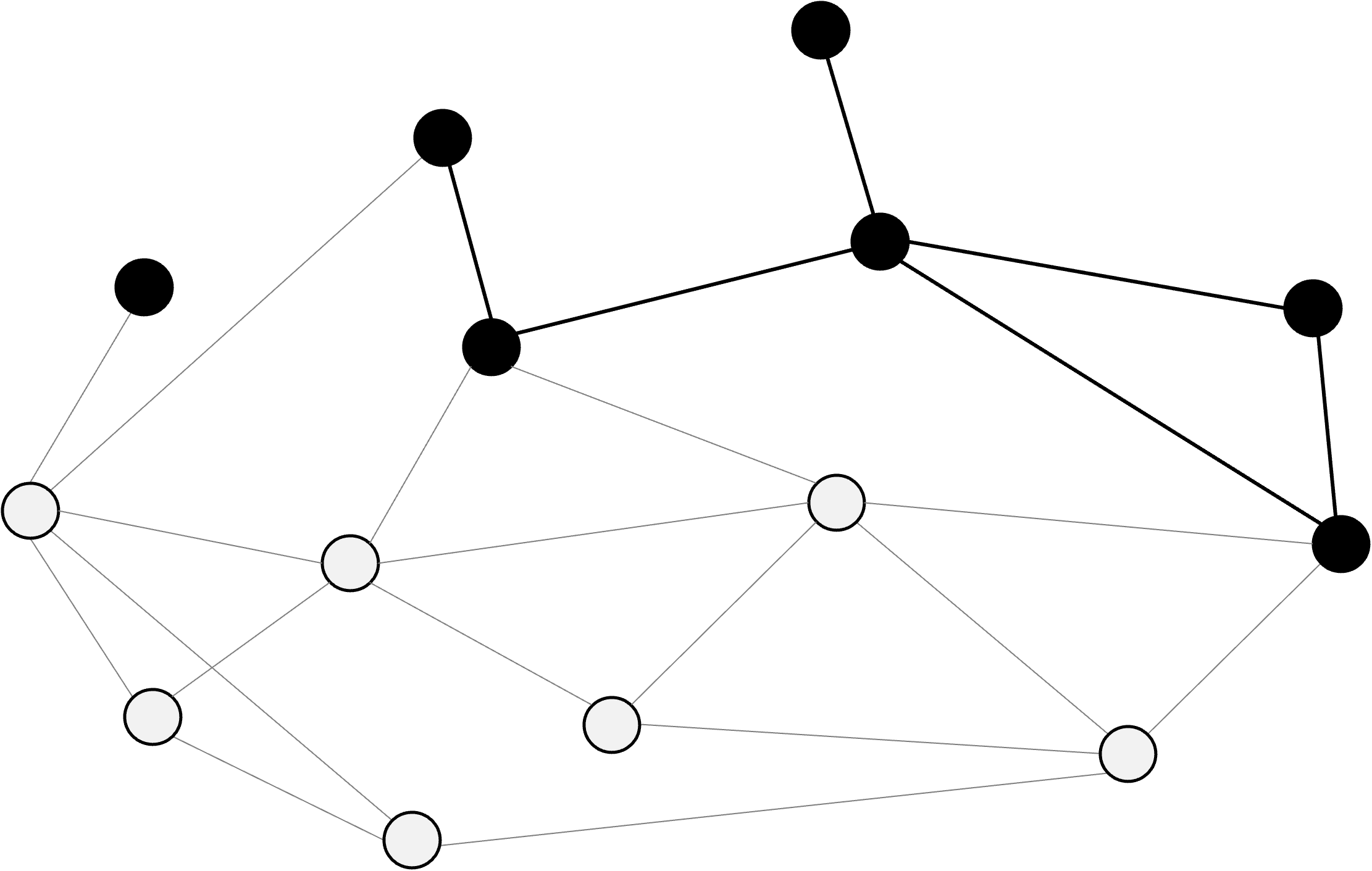}
  \caption{Nodes and edges in the subgraph induced by the nodes outside of the $3$-core of graph in Figure~\ref{fig:example_coreness}.}
  \label{fig:kcore2}
\end{figure}

In Figure~\ref{fig:kcore2}, using the graph in Figure~\ref{fig:example_coreness} with $k=3$, we report in black the nodes $i$ for which $u_i=1$ and the edges $\{i,j\}$ for which $x_{ij}=1$ in the optimal solution of formulation~\eqref{eq:kcore2}, i.e., the nodes and the edges of the subgraph induced by the nodes outside of the $3$-core.

We now claim that the optimal solution of the continuous relaxation of the model~\eqref{eq:kcore2-mccormick} (i.e., the formulation obtained by replacing constraints~\eqref{eq:kcore2-sets} with $u\in[0,1]^n, \; x\in[0,1]^{|E|}$) is integer.

\begin{theorem}\label{th:integr}
Any optimal solution $(x^*,u^*)$ of the continuous relaxation of formulation~\eqref{eq:kcore2-mccormick} is integer.
\end{theorem}
\begin{proof}
Any optimal solution of the continuous relaxation of formulation~\eqref{eq:kcore2-mccormick} is feasible, i.e., it satisfies the following constraints 
\begin{equation}\label{constr_opt}
  \sum\limits_{j \in N(i)} u^*_j + (k - |N(i)|) \leq \sum\limits_{j \in N(i)}x^*_{ij} + (k - |N(i)|)u^*_i, \; \forall\,i \in V.
\end{equation}
For a certain node $i$, let $N_1(i)$ be the set of nodes $j \in N(i)$ such that $u_j^*=1$. For the nodes in the set $N_1(i)$, $x_{ij}^*=u_i^*$ because of constraints~\eqref{eq:kcore2-mccormick1} and \eqref{eq:kcore2-mccormick3}. We can thus write 
$$\sum\limits_{j \in N(i)} u^*_j = \sum\limits_{j \in N_1(i)} u^*_j + \sum\limits_{j \in N(i)\backslash N_1(i)} u^*_j = |N_1(i)| + \sum\limits_{j \in N(i) \backslash N_1(i)} u^*_j,$$
and 
$$\sum\limits_{j \in N(i)}x^*_{ij} = \sum\limits_{j \in N_1(i)}u^*_{i}+ \sum\limits_{j \in N(i)\backslash N_1(i)}x^*_{ij} = |N_1(i)|u^*_{i}+ \sum\limits_{j \in N(i)\backslash N_1(i)}x^*_{ij} \leq |N_1(i)|u^*_{i}+ \sum\limits_{j \in N(i)\backslash N_1(i)}u^*_{j},$$
the last inequality coming from the McCormick constraint~\eqref{eq:kcore2-mccormick2}.
We can thus reformulate~\eqref{constr_opt} as:
\begin{equation}
  |N_1(i)| + \sum\limits_{j \in N(i) \backslash N_1(i)} u^*_j + (k - |N(i)|) \leq |N_1(i)|u^*_{i}+ \sum\limits_{j \in N(i)\backslash N_1(i)}u^*_{j} + (k - |N(i)|)u_i^*,
\end{equation}
which can be reduced to 
\begin{equation}
  \Big[|N_1(i)| + (k - |N(i)|)\Big](1-u_i^*) \leq 0.
\end{equation}
If $|N_1(i)| + (k - |N(i)|) > 0$, then $u_i^*=1$.
Otherwise, if $|N_1(i)| + (k - |N(i)|) \leq 0$ $u_i^*$ can have any value in $[0,1],$ but since we are minimizing, and reducing the value of $u_i$ does not impact on the feasibility of McCormick constraints, it will be set to 0.
Thus, in the optimal solution of the continuous relaxation of \eqref{eq:kcore2-mccormick}, each $u_i$ has value equal to either 0 or 1, thus the solution is integer.
\end{proof}

Thus, the linear relaxation of problem~\eqref{eq:kcore2-mccormick} provides the optimal solution of formulation~\eqref{eq:kcore2}. 
We can further notice that constraints~\eqref{eq:kcore2-mccormick3}, and $x \leq 1$, can be dropped. Indeed, they are redundant and guaranteed by the remaining constraints.
The resulting relaxed problem is
\begin{subequations}\label{eq:kcore2-lin}
\begin{align}
  \max_{x,u} &\; n - \sum_{i \in V} u_i & \label{eq:kcore2-lin-obj}\\
  \text{s.t.} & \sum_{j \in N(i)} u_j + k - |N(i)| \leq \sum\limits_{j \in N(i)}x_{ij} + (k - |N(i)|)u_i &\forall\,i \in V \label{eq:kcore2-lin0}\\ 
  &\; x_{ij} \leq u_i & \forall\,i \in V, j \in N(i) \label{eq:kcore2-lin1}\\
  &\; x_{ij} \leq u_j & \forall\,i \in V, j \in N(i)\label{eq:kcore2-lin3}\\
  &\; u \in [0,1]^n, \; x\in\mathbb{R_+}^{|E|}
\end{align}
\end{subequations}
which is linear in the variables $x$ and $u$. 

\section{Mathematical formulations for the Collapsed $k$-Core Problem}\label{sec:math-formulations}
Leveraging on the formulations presented above for the $k$-Core Detection Problem, in this section, we propose three models for the Collapsed $k$-Core Problem, which is defined as follows.
\begin{definition}
Given an undirected graph $\mathcal{G}=(V,E)$ and two positive integers $k$ and $b$,
the Collapsed $k$-Core Problem consists of finding a subset $W^*\subset V$ of $b$ nodes, 
the removal of which minimizes the size of the resulting $k$-core (i.e., for which $|C_k(\mathcal{G}\backslash W^*)|$ is minimum).
\end{definition} 

In the following, we will refer to these $b$ nodes as \textit{collapsers}.

\begin{figure}[h!]
  \centering
  \includegraphics[width=0.5\textwidth]{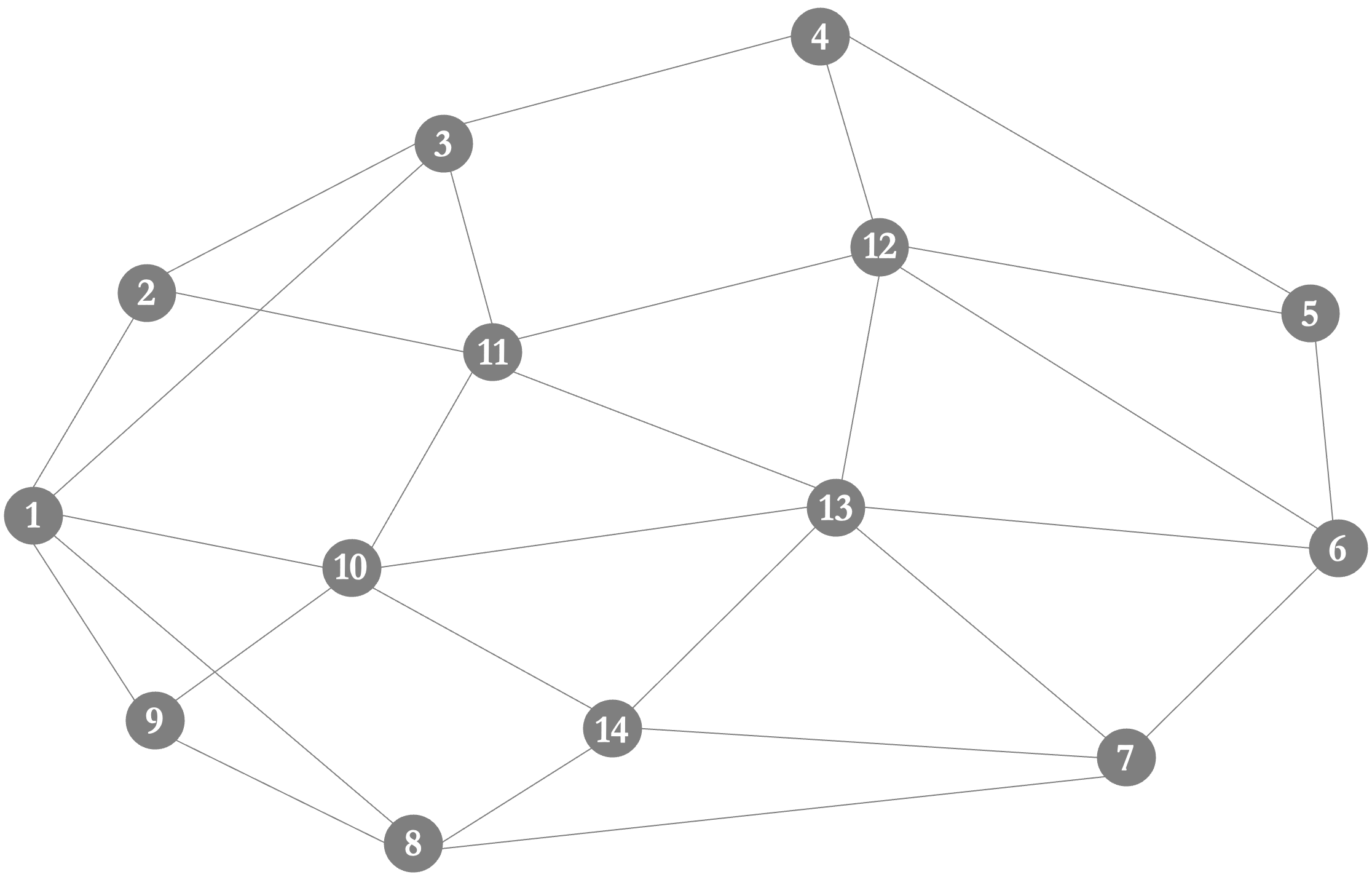}
  \caption{\textbf{Example graph.}}
  \label{fig:starting_graph}
\end{figure}
In Figures~\ref{fig:starting_graph}, \ref{fig:example_1} and \ref{fig:example_2} some of the introduced concepts are visualized through examples. The graph $\mathcal{G}$ with $14$ nodes in Figure \ref{fig:starting_graph} is a $3$-core. Indeed, every node has degree at least $3$. The subgraph induced by the set $S=\{6,7,10,11,12,13,14\}$ is a $3$-subcore, i.e., $S$ is a set of nodes such that the degree of $\mathcal{G}[S]$ is three, but its cardinality is not maximum. 
Assume that we want to determine the Collapsed 3-core Problem on this graph with budget $b=1$, i.e., we can remove one node only. Two possible feasible solutions of this Collapsed $3$-core Problem are represented in Figure~\ref{fig:example_1} and Figure~\ref{fig:example_2}. The first one in Figure~\ref{fig:example_1}, which consists in removing node $13$ from the graph, leads to a $3$-core of cardinality $13$ (no node follows node $13$, because no node has less than $3$ neighbors in the remaining graph). The second one in Figure~\ref{fig:example_2}, which consists in removing node $1$ from the graph, leads to a better solution, since the obtained $3$-core after the cascade effect following node $1$ removal consists in $7$ nodes. In fact, this is the optimal solution of the Collapsed $3$-core Problem for the graph in Figure \ref{fig:starting_graph}. 
\begin{figure}[H]
  \vspace*{-0.8cm}
   \centering
   \hspace*{\fill}%
   \hfill
   \subfloat[Selection of node 13 to remove.]{\includegraphics[scale=0.36]{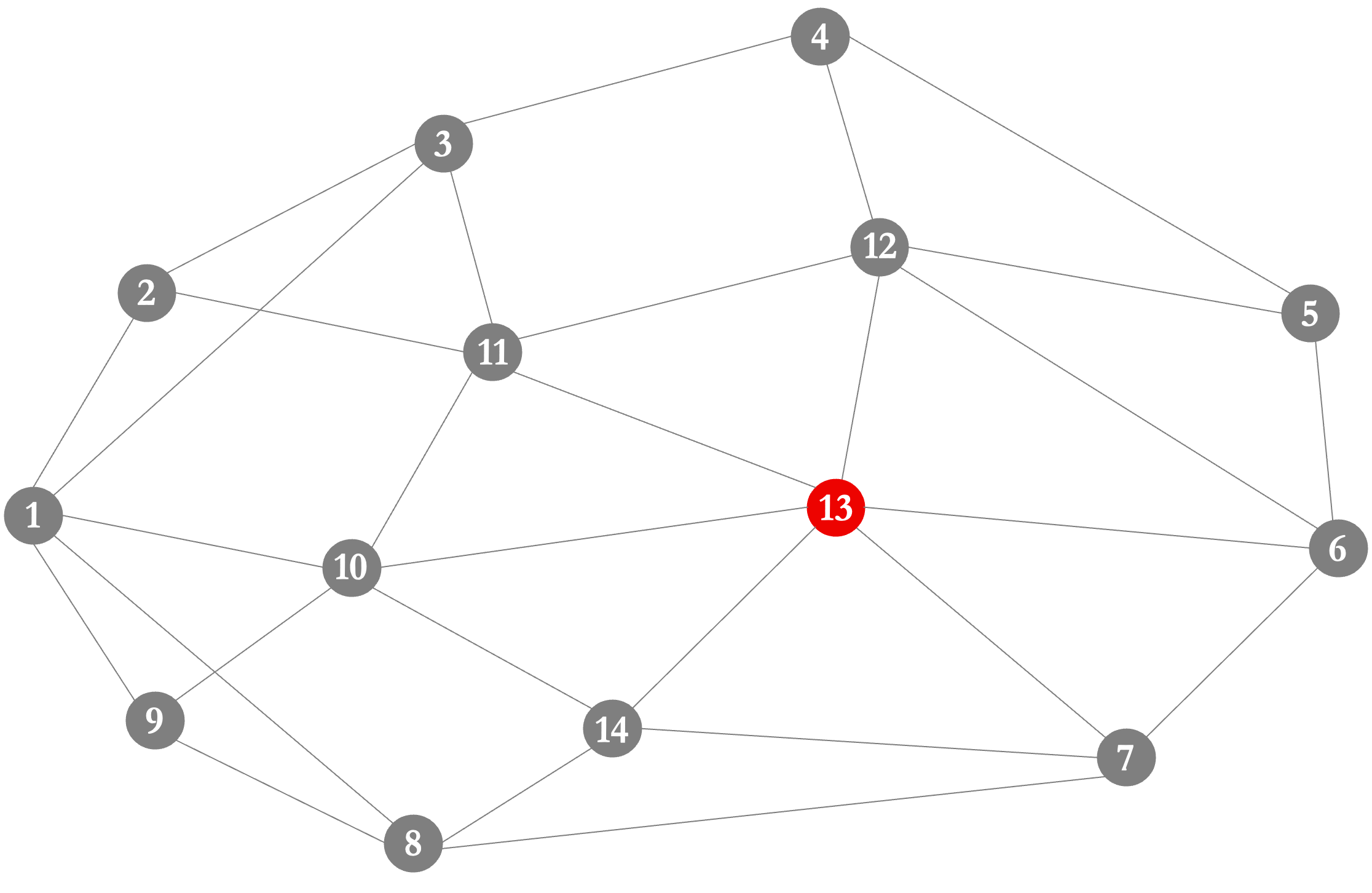}\label{ex1:0}}\hfill
  \subfloat[Node 13 leaving.]{\includegraphics[scale=0.36]{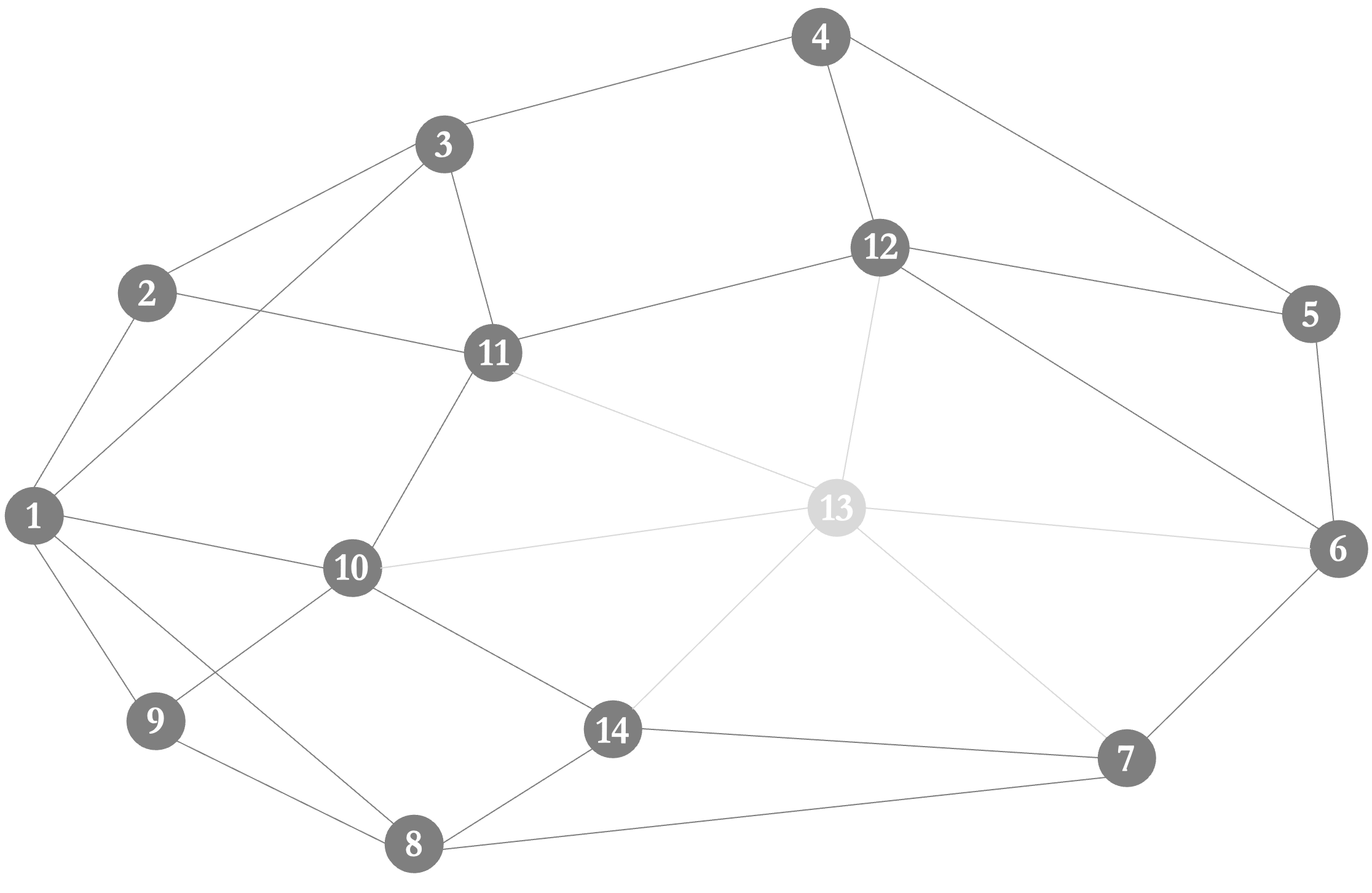}\label{ex1:1}}\hfill
  \hspace*{\fill}%
  \caption{\textbf{Feasible and suboptimal solution.}}
  \label{fig:example_1}

\end{figure}

\begin{figure}[h!]
  \centering
  \hspace*{\fill}%
  \hfill
  \subfloat[Selection of node 1 to remove.]{\includegraphics[width=.38\textwidth]{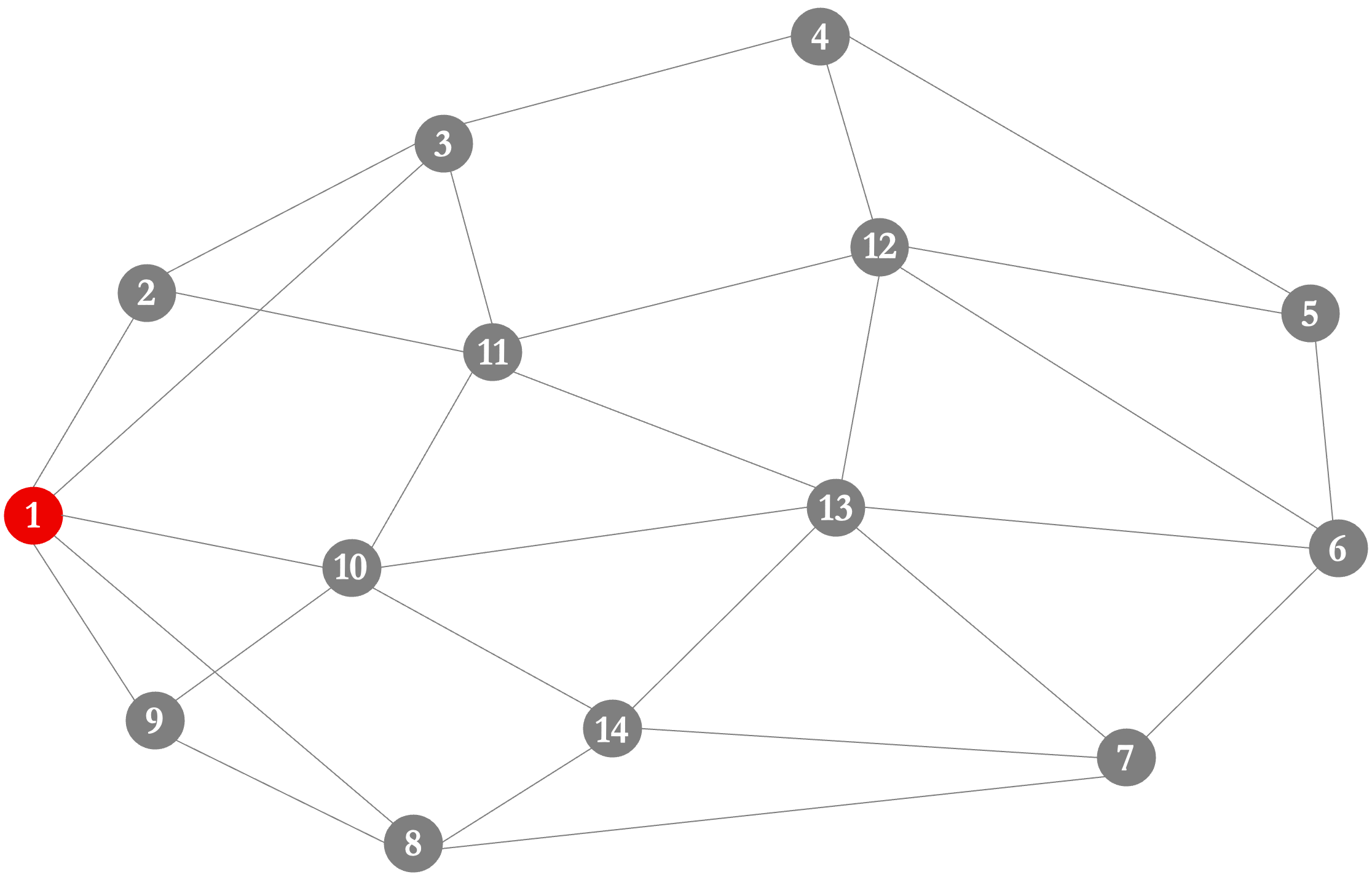}\label{ex2:0}}\hfill
  \subfloat[Node 1 leaving.]{\includegraphics[width=.38\textwidth]{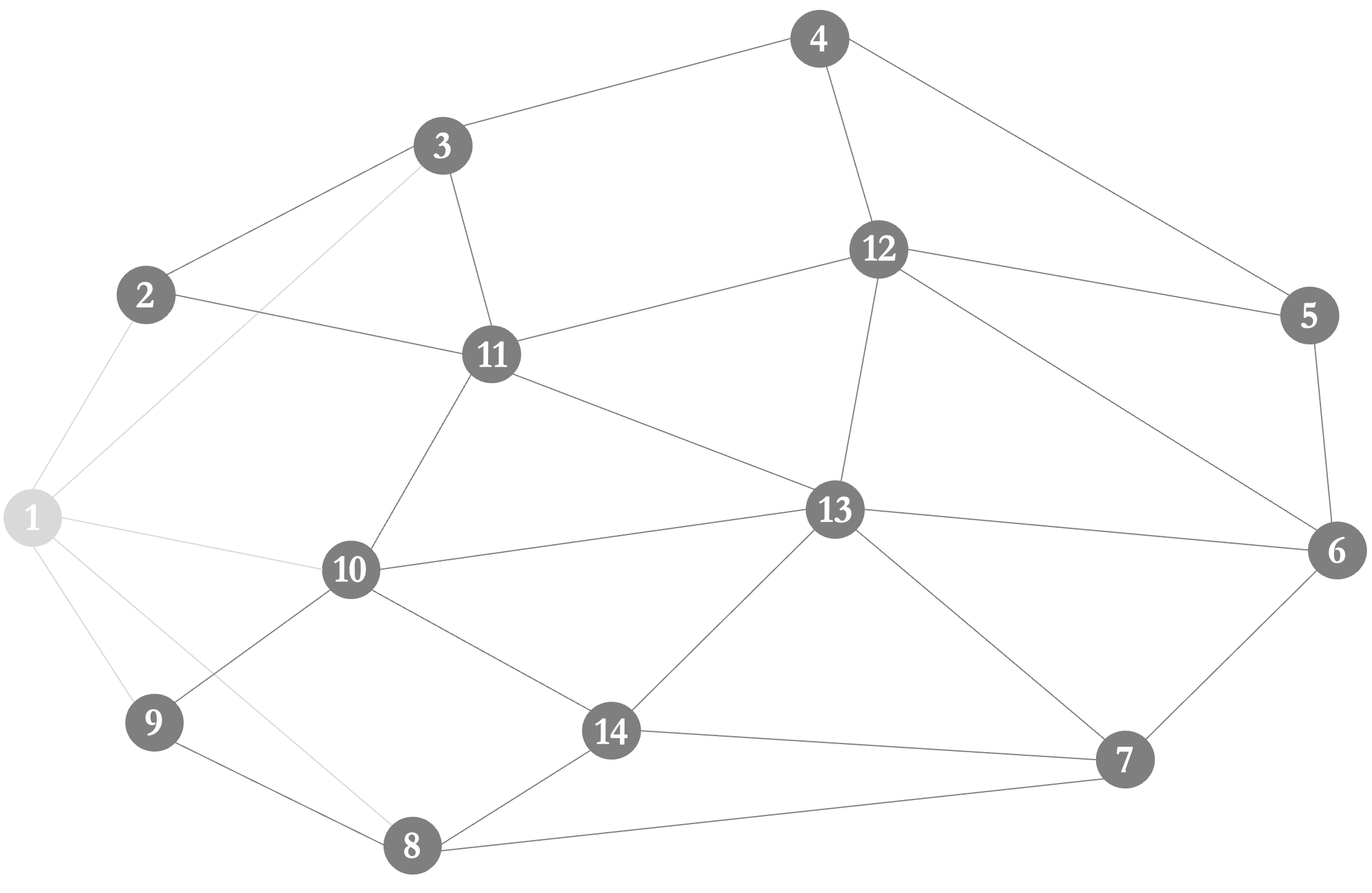}\label{ex2:1}}\hfill
  \hspace*{\fill}\\
  \hspace*{\fill}%
  \hfill
  \subfloat[Nodes 2 and 9 leaving.]{\includegraphics[width=.38\textwidth]{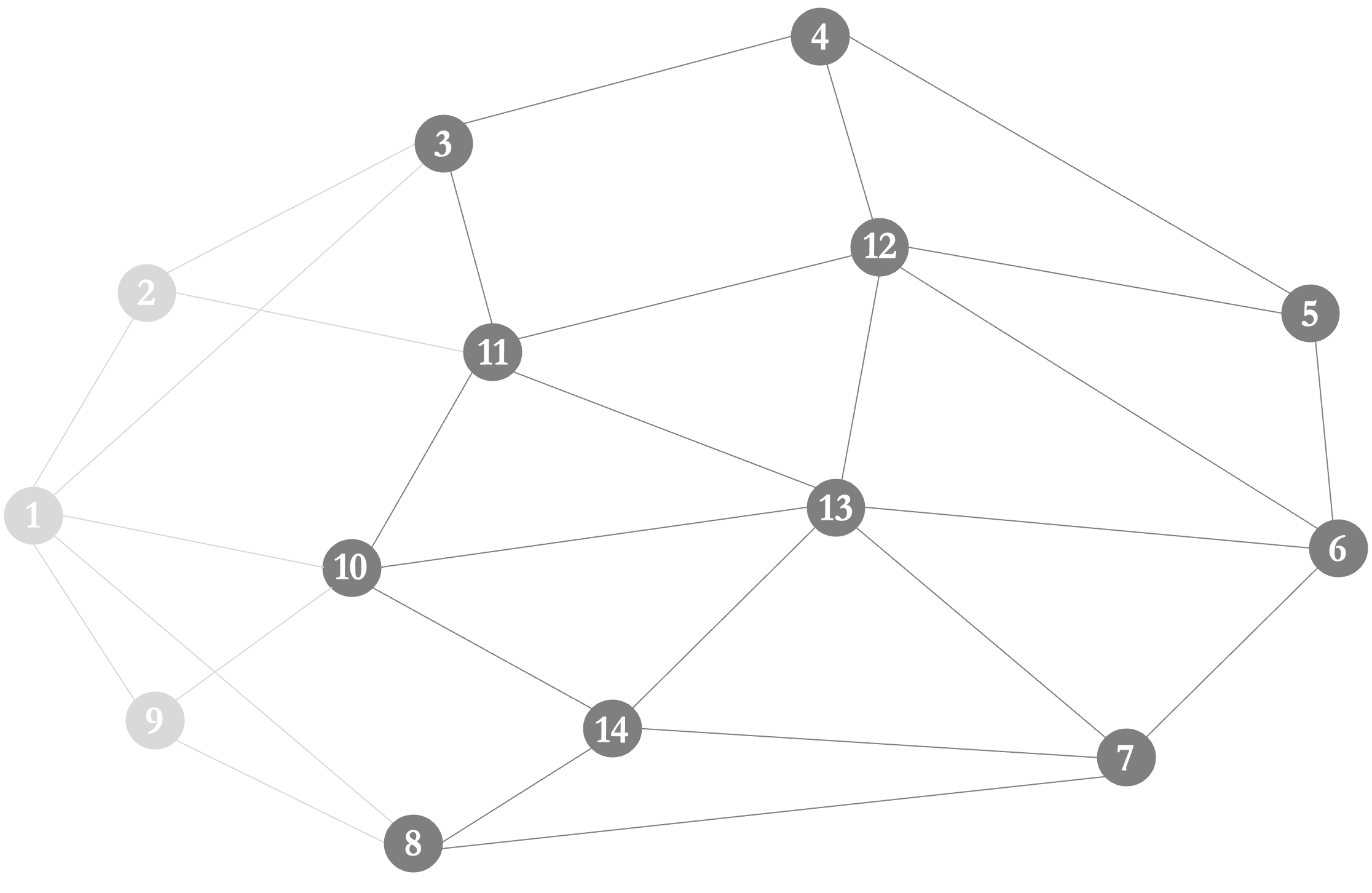}\label{ex2:2}}\hfill
  \subfloat[Nodes 3 and 8 leaving.]{\includegraphics[width=.38\textwidth]{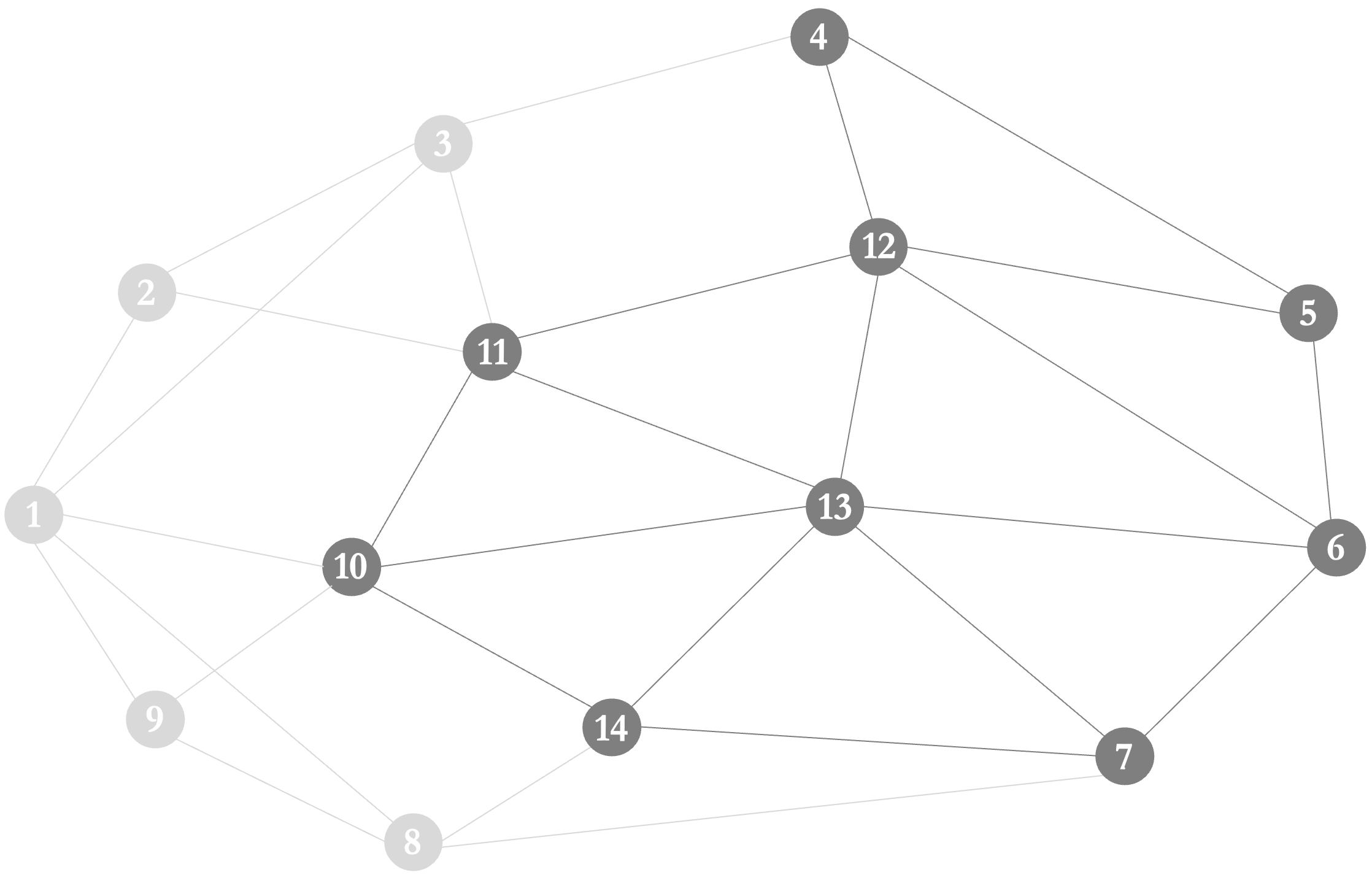}\label{ex2:3}}\hfill
  \hspace*{\fill}\\
  \hspace*{\fill}%
  \hfill
  \subfloat[Node 4 leaving.]{\includegraphics[width=.38\textwidth]{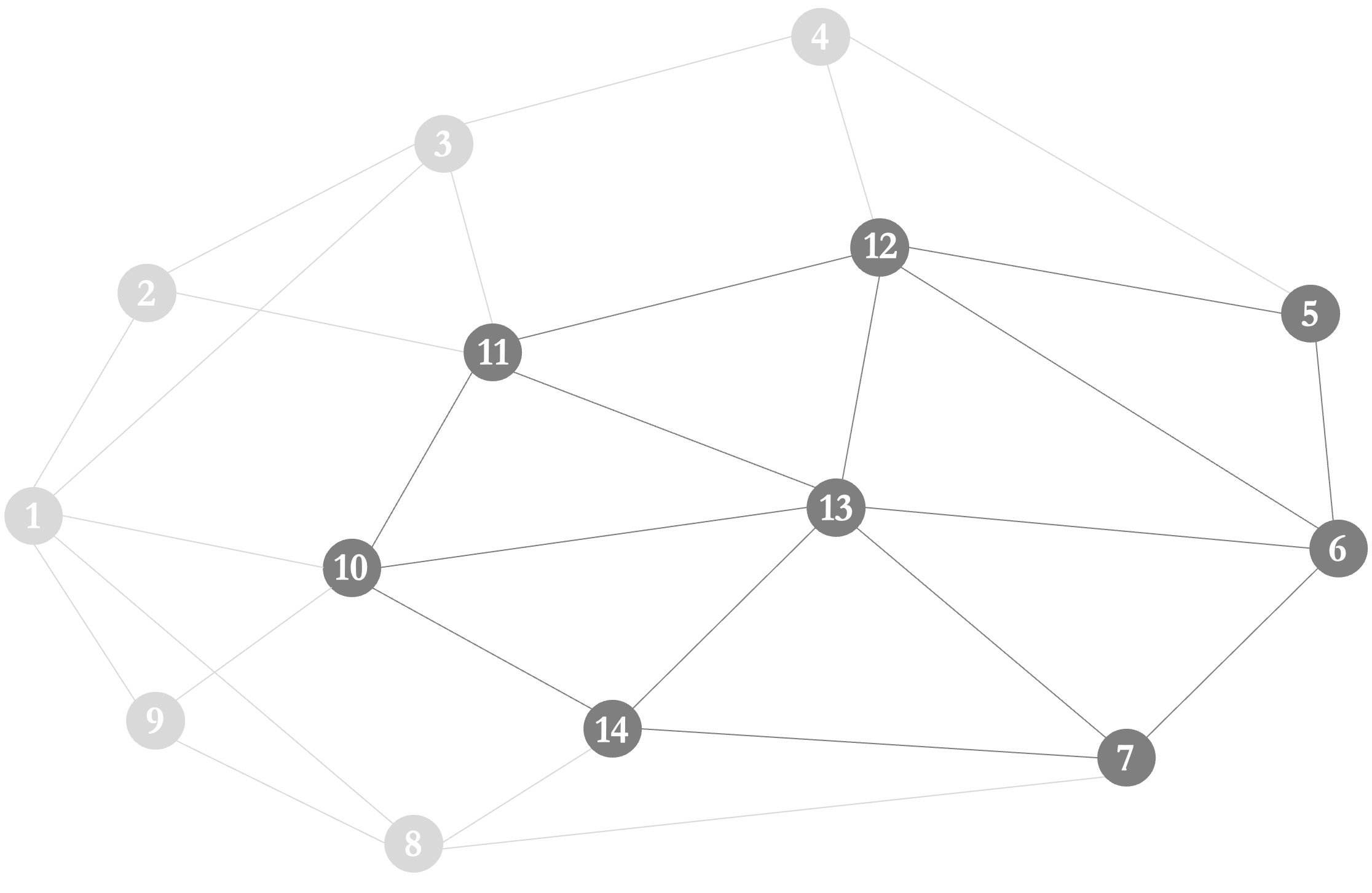}\label{ex2:4}}\hfill
  \subfloat[Node 5 leaving.]{\includegraphics[width=.38\textwidth]{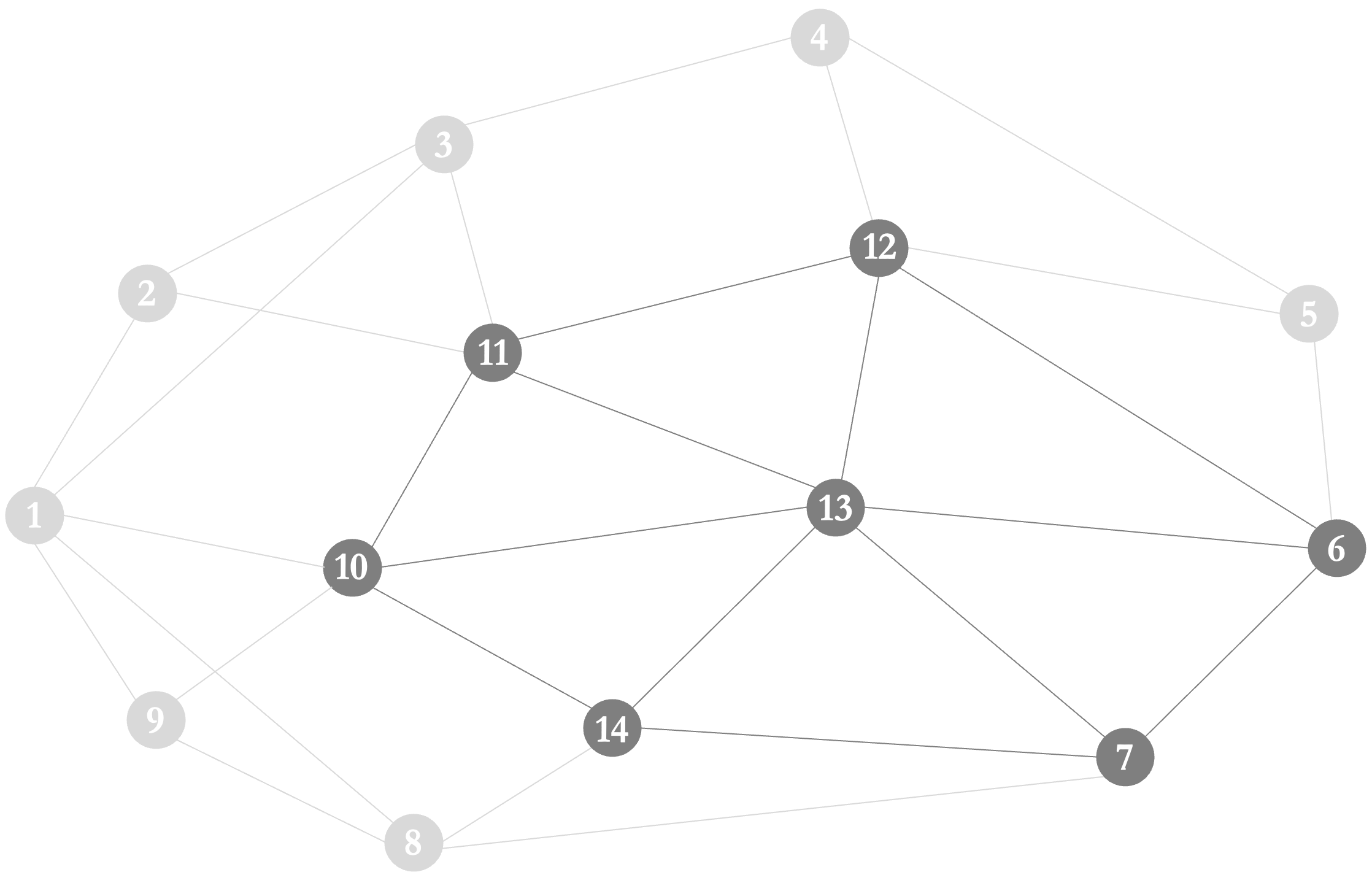}\label{ex2:5}}\hfill
  \hspace*{\fill}\\%
  \caption{\textbf{Optimal solution.}}\vspace{-0.5cm}
  \label{fig:example_2}
\end{figure}

In the rest of this section, we propose three formulations for the problem: the first one relies on the iterative process used to determine the collapsed $k$-core, after the removal of the $b$ nodes; the other two formulations are bilevel formulations, considering two agents, a leader who selects the nodes to remove and a follower who computes the resulting $k$-core, solving either formulation \eqref{eq:kcore1} or \eqref{eq:kcore2-lin}.
The bilevel structure of the problem comes from the fact that the $k$-core is defined as the induced subgraph with all nodes having degree at least $k$ of maximum size. Given that, the aim of the problem is to find the set of $b$ nodes to remove, such that the maximal induced $k$-subcore is of minimum cardinality.

\subsection{Time-Dependent Formulation}\label{subsec:timedep}
In this section, we describe a ``natural'' ILP formulation for the Collapsed $k$-Core Problem. This formulation models the so-called cascade effect deriving from the removal of $b$ nodes, which are removed at time $0$. At each deletion round (corresponding to a time instant $t$), all the nodes the degree of which becomes less than $k$ are removed from the graph.
Assuming that the problem instance is feasible, i.e., that the considered graph has at least $b$ nodes, the deletion rounds are at most $T = n-b$. 

We introduce the binary variables $a_i^t$, defined for each node $i \in V$ and time $t=0,\dots,T$, such that:
\begin{displaymath}
	a_{i}^t = \begin{cases}
		1 &\text{if node $i$ belongs to the induced subgraph at time $t$}\\
		0 &\text{if node $i$ does not belong to the induced subgraph at time $t$}
	\end{cases}
\end{displaymath}
In the example shown in Figure~\ref{fig:example_2}, e.g., $a_{1}^0 = 0$ and $a_{i}^0 = 1$ for $i \in V\backslash\{1\},$ while, as shown in Figure~\ref{ex2:5}, $a_{i}^T = 1$ for $i=\{6,7,10,11,12,13,14\}$ and $a_{i}^T = 0$ for $i=\{1,2,3,4,5,8,9\}$, where $T=4$.

The Collapsed $k$-Core Problem can be formulated as:
 \begin{subequations}\label{eq:formulation_td}  
 \begin{align}
    \min\limits_{a} &\, \sum_{i \in V} a_i^T & \label{eq:objfunction_td}\\
    \text{s.t.} & \, \sum_{i \in V} a_{i}^0 = n - b & \label{eq:budget_td}\\
    & \;\; a_{i}^t \le a_{i}^{t-1} & \forall\,i \in V,\; t = 1, \dots, T \label{eq:onlyremoval_td}\\
    &\, \sum_{j \in N(i)} a^{t-1}_j - k +1 \leq d(i)\revised{\cdot}a^t_i + d(i)\revised{\cdot}(1 - a^0_i) & \forall\,i \in V,\; t = 1, \dots, T \label{eq:const_one_td} \\
    &\; a_i^t \in \{0,1\} & \forall\,i \in V,\; t = 0, \dots, T
  \end{align}
\end{subequations}
where $d(i)=|N(i)| - k + 1$.
The objective function minimizes the number of nodes remaining in the last time instant $T$. The first constraint~\eqref{eq:budget_td} imposes that at the beginning (first time instant) exactly $b$ nodes are removed. Constraints~\eqref{eq:onlyremoval_td} state that, if node $i$ is not in the graph at time~$t-1$, it cannot be in the graph at time~$t$.
Finally, constraints~\eqref{eq:const_one_td} ensure that, if the node is not removed at time 0, it must stay in the graph at time $t$ iff more than $k-1$ of its neighbors ``survived'' at time $t-1$. Constraint~\eqref{eq:const_one_td}, for a given $i\in V$, and a given $t \in \{1,\dots,T\}$, imposes that
\begin{itemize}
  \item if the node is not a collapser, i.e., it is not removed at time 0, and thus $a_i^0=1$
  \item and $ {\sum\limits_{j \in N(i)} a^{t-1}_j \geq k}$, i.e., $\sum\limits_{j \in N(i)} a^{t-1}_j - k + 1 \geq 1 $ (with $\sum\limits_{j \in N(i)} a^{t-1}_j \leq |N(i)|$)
\end{itemize}
then $a_i^t$ is set to 1. The term $(|N(i)| - k + 1)(1 - a^0_i)$ is thus needed to guarantee that, for the collapsers, the constraints~\eqref{eq:const_one_td} are still satisfied. For these nodes indeed $a_i^0$ and $a_i^t$ will be 0.

The obtained ILP formulation is polynomial in size and can be solved using existing state-of-the-art solvers. It is indeed a compact formulation, which is the main advantage of this first natural approach. However, because of time index $t$, the number of variables is large, thus requiring a long computational time to be solved. For this reason, we present in the following sections two alternative formulations for the Collapsed $k$-Core Problem.

\subsection{A first bilevel formulation}
In this section, we formulate the Collapsed $k$-Core Problem using bilevel programming where the leader aims at \emph{minimizing} the cardinality of the $k$-core obtained by removing exactly $b$ nodes. The follower instead aims at detecting the $k$-core obtained after the removal of the $b$ nodes chosen by the leader, which corresponds to finding the \emph{maximal} induced subgraph where all the nodes have degree at least $k$. The follower's problem is modeled as in \eqref{eq:kcore1}, with additional linking constraints imposing that the $k$-core is computed in the graph resulting after the removal of $b$ nodes by the leader.

We consider the upper-level binary variables $w_i$ and the lower-level binary variables $y_i$ (already defined in Eq.~\eqref{eq:kcore1-var}), both defined for each node $i \in V$:
	
	\begin{displaymath}
		w_i = \begin{cases}
			1 &\text{if node $i$ is removed by the leader}\\
			0 &\text{otherwise}
		\end{cases}
	\end{displaymath}
	
	\begin{displaymath}
		y_i = \begin{cases}
			1 &\text{if node $i$ is in the collapsed $k$-core of the follower}\\
			0 &\text{otherwise}
		\end{cases}
	\end{displaymath}
\begin{subequations}

Variable $w_i$ is 1 iff node $i$ is a collapser.
We remark that it corresponds to $1-a_i^0$ in the time-dependent formulation~\eqref{eq:formulation_td}, presented in Subsection~\ref{subsec:timedep}. Variable $y_i$, instead, is 1 iff node $i$ is in the resulting $k$-core, thus corresponds to $a_i^T$ in the time-dependent formulation~\eqref{eq:formulation_td}.

In the example shown in Figure~\ref{fig:example_2}, e.g., $w_1 = 1$ and $w_{i} = 0$ for $i \in V\backslash\{1\},$ while, as shown in Figure~\ref{ex2:5}, $y_{i} = 1$ for $i=\{6,7,10,11,12,13,14\}$ and $y_{i} = 0$ for $i=\{1,2,3,4,5,8,9\}$.

The set of all possible leader's policies $\mathcal{W}$ (removing exactly $b$ nodes) is
	\begin{equation}\label{eq:Wset}
		\mathcal{W} = \left\{w \in \{0,1\}^n : \sum_{i \in V} w_i = b \right\}.
	\end{equation}
Let ${\Omega}$ denote the set of all subsets $W \subseteq V$ such that $|W|=b$, \revised{i.e., $\Omega = \{W\subseteq V ~:~|W| = b\}$}. There is a one-to-one correspondence between each element $W \in {\Omega}$ and its incidence vector $w \in \mathcal{W}$. The set of all possible nodes subsets inducing $k$-subcores of $\mathcal{G}$ is $\mathcal{K},$ defined in Subsection~\ref{subsec:kcore1}.
If no node is interdicted by the leader, the problem of the follower corresponds to finding the $k$-core in the original graph, defined in Subsection~\ref{subsec:kcore1}.
\end{subequations}
The resulting Collapsed $k$-Core Problem can be formulated as the following bilevel problem:
	\begin{equation}\label{eq:formulation_bil}
		\min_{w \in \mathcal{W}} \max_{y \in \mathcal{Y}} \left\{ \sum_{i \in V} y_i~:~y_i \le 1-w_i,\; \forall\,i \in V\right\}.
	\end{equation}
Constraints $y_i \le 1-w_i$ exclude from the $k$-core the collapsers. Problem~\eqref{eq:formulation_bil} does not consider the optimal solutions of the follower's problem, but only its optimal objective function value. Thus, we do not need to distinguish between optimistic and pessimistic concepts. 

\subsubsection{
Sparse Formulation}\label{subsubsec:sparse}
Formulation~\eqref{eq:formulation_bil} exhibits the structure of so-called \emph{interdiction} problems, which is a well-known class of bilevel optimization problems (see~\citet{ivana_survey,Smith-Song:2020} for recent surveys on interdiction problems).
This structure allows us to apply a Benders-like reformulation technique in which we project out the lower-level variables, and introduce an auxiliary
integer variable $z$ to represent the objective value of the lower-level problem. We refer to this formulation as \emph{sparse}, because it is given in the natural space of $w$ variables (required to describe the removed set of nodes) and a single auxiliary variable. 
We start by reformulating the problem \eqref{eq:formulation_bil} as follows:
\begin{subequations}\label{eq:formulation_bil1}
	\begin{align}
		\min_{z \in \mathbb{Z},w \in \{0,1\}^n} &\;z\\
		\text{s.t.} &\;z \geq \max_{y \in \mathcal{Y}} \left\{ \sum_{i \in V} y_i \,:\,y_i \le 1-w_i, \forall\,i \in V\right\} \label{eq:bilconstr}\\
		&\; \sum_{i \in V} w_i = b. \label{eq:budget_bil}
	\end{align}
\end{subequations}

Following the ideas from e.g., \citet{Wood:2011,FischettiLMS19,Leitner-et-al:2022}, we can then reformulate~\eqref{eq:bilconstr}, given sufficiently large $M_i$ for all $i$, as:
\begin{align}
z \geq \max_{y \in \mathcal{Y}} \left\{ \sum_{i \in V} y_i - \sum_{i \in V} M_iy_i w_i \right\} \label{eq:lowerlevel}
\end{align}
which is equivalent to the following Benders-like constraints:
\begin{align}
	z \geq \sum_{i \in V} \hat y_i - \sum_{i \in V} M_i w_i \hat y_i \quad \forall\, \hat y \in \mathcal{Y}. \label{eq:theta0}
\end{align}

\noindent In terms of sets $K \in \mathcal{K}$ inducing the $k$-subcores of $\mathcal{G}$, we can rewrite Ineqs.~\eqref{eq:theta0} as
\begin{align}
	z \geq | K| - \sum_{i \in K} M_i w_i \quad \forall\, K \in \mathcal{K}. \label{eq:theta1}
\end{align}
The value of $M_i$ needs to be set in such a way that, for a given $K$, in case node $i$ (possibly together with some other nodes from $K$) is interdicted, the value $|K| - M_i$ gives the lower bound on the size of the $k$-core of $\mathcal{G}[K]$. Since, in the extreme case, the interdiction can lead to an empty $k$-core, we set $M_i = |K|$, resulting into 
\begin{align}\label{eq:bigM_total}
	z \geq | K| \left[ 1 - \sum_{i \in K} w_i \right] \quad \forall\, K \in \mathcal{K}.
\end{align}

Alternatively, since the upper-level decisions are binary, we can reformulate constraint~\eqref{eq:bilconstr} using the following $\binom{n}{b}$ many no-good-cuts: 
\begin{equation}\label{eq:nogoodcut_total}
z \ge \left| C_k(\mathcal{G}\backslash W)\right|\left[ \sum_{i \in W} w_i - b + 1 \right], \qquad 
\forall\, W \in \Omega.
	\end{equation} 
Indeed, for any given $W$, 
the cardinality of 
the $k$-core $C_k(\mathcal{G}\backslash W)$, when the nodes from $W$ are removed by the leader, provides a valid lower bound on $z$. If at least one of the nodes in $W$ is not a collapser, the related constraint of type \eqref{eq:nogoodcut_total} turns out to be redundant, since the right hand side becomes less than or equal to 0.

In Subsection~\ref{subsec:separation_bil}, we will present a separation procedure to solve the single-level reformulation of \eqref{eq:formulation_bil1} given as 
\begin{subequations}\label{eq:singlewithcuts}
	\begin{align}
		\min_{z \in \mathbb{Z},w \in \mathcal{W}} &\;z\\
		\text{s.t. \eqref{eq:bigM_total}, \eqref{eq:nogoodcut_total}} 
	\end{align}
\end{subequations}
which has exponentially many constraints.

\subsection{A second bilevel formulation}
In this section, we propose an alternative bilevel formulation by considering lower-level variables $u_i$ defined in \eqref{eq:kcore2-var}, complementary with respect to\ variables $y_i$, and formulating the lower-level problem as in 
\eqref{eq:kcore2-lin}.

We recall that $u_i$ represent the lower-level variables identifying the nodes not belonging to the collapsed $k$-core of the graph:
\begin{displaymath}
	u_i = \begin{cases}
		1 &\text{if node $i$ does not belong to the collapsed $k$-core of the follower}\\
		0 &\text{otherwise}
	\end{cases}
\end{displaymath}
We remark that variable $u_i$ corresponds to $1-y_i$ in the former bilevel formulation. 
Indeed, in the example shown in Figure~\ref{fig:example_2}, $u_{i} = 1$ for $i=\{1,2,3,4,5,8,9\}$, and $u_{i} = 0$ for $i=\{6,7,10,11,12,13,14\}$.

The problem of detecting the $k$-core of the graph, given a leader's decision $w \in \mathcal{W}$, can be modelled as the problem of determining the set of nodes outside of the $k$-core:
\begin{subequations}\label{eq:degen}
\begin{align}
 \Psi(w) := 
 \min_{u} &\; \sum_{i \in V} u_i & \label{eq:obj-deg}\\
 \text{s.t.} & \; \sum_{j \in N(i)} u_j + k - |N(i)| \leq k\revised{\cdot}u_i &\forall\,i \in V \label{eq:degeneracy}\\ 
 &\;w_i \leq u_i & \forall\,i \in V\label{eq:interdiction}\\
 & \; u \in \{0,1\}^{n} & 
\end{align}
\end{subequations}

This formulation differs from formulation~\eqref{eq:kcore2} only for constraints~\eqref{eq:interdiction}, which state that a node $i$ cannot be in the $k$-core if it is interdicted/removed. Note that problem~\eqref{eq:degen} may be equivalently written as a maximization problem, with objective function $n - \sum_{i \in V} u_i$, as in \eqref{eq:kcore2-obj}.

The corresponding bilevel formulation of the Collapsed $k$-Core Problem is:
\begin{subequations}\label{eq:formulation_bil2}
\begin{align}
 \min_{ v \in \mathbb{Z},w\in \{0,1\}^{n}} & \; n - v \label{eq:obj-deg-bil}\\
 \text{s.t.} & \; v \leq \Psi(w) \label{eq:ll-deg-bil}\\
 &\sum\limits_{i \in V} w_i = b. \label{eq:budgetbil2}
\end{align}
\end{subequations}
The objective function expresses the fact that we want to find the minimal collapsed $k$-core, computed by solving $\Psi(w).$

In Subsection~\ref{subsec:kcore2}, we proved that problem~$\Psi(w)$ can be reformulated as the following LP formulation:
\begin{subequations}\label{eq:LP}
\begin{align}
 \min_{u} &\; \sum_{i \in V} u_i & \label{eq:lin-obj}\\
 \text{s.t.} & \sum_{j \in N(i)} u_j + k - |N(i)| \leq \sum\limits_{j \in N(i)}x_{ij} + (k - |N(i)|)u_i &\forall\,i \in V \label{eq:lin-1}\\ 
 &\; x_{ij} \leq u_i & \forall\,i \in V, j \in N(i) \label{eq:lin-2}\\
 &\; x_{ij} \leq u_j & \forall\,i \in V, j \in N(i)\label{eq:lin-3}\\
 &\; u_i \geq w_i & \forall\,i \in V \label{eq:lin-4}\\
 &\; u \in [0,1]^n, \; x\in\mathbb{R_+}^{|E|}
\end{align}
\end{subequations}
The addition of constraints~\eqref{eq:lin-4}, indeed, has no impact on the proof of Theorem~\ref{th:integr}. Since formulation~\eqref{eq:LP} is linear in the variables $x$, and $u$, we can replace it by its dual, as detailed in the following section. 

\subsubsection{Compact nonlinear formulation}\label{subsubsec:dual}
An approach to deal with the bilevel formulation \eqref{eq:formulation_bil2} consists in dualizing the lower level continuous formulation~\eqref{eq:LP}. Let us define the following dual variables for all $i \in V$: $\alpha_i$ associated with the constraints~\eqref{eq:lin-1}; $\beta_{ij}, \forall\,j \in N(i)$ associated with the constraints~\eqref{eq:lin-2}; $\gamma_{ij}, \forall\,j \in N(i)$ associated with the constraints~\eqref{eq:lin-3}; $\lambda_i$ associated with the constraints~\eqref{eq:lin-4}; $\tau_i$ associated with the constraints $u_i\leq 1.$

The dual of the lower-level problem \eqref{eq:LP} is:
{\small\begin{subequations}\label{eq:dualref_pre}
	\begin{align}
  &	\max_{\alpha,\beta,\gamma,\lambda,\tau} \;\sum_{i \in V}\left[(k-|N(i)|)\alpha_i +w_i\lambda_i-\tau_i\right] \\
  & (k-|N(i)|)\alpha_i + \lambda_i - \tau_i +\sum_{j \in N(i)}(-\alpha_j +\beta_{ij} + \gamma_{ji}) \leq 1 \quad \forall\,i \in V \label{eq:dual1}\\ 
 & \alpha_i - \beta_{ij} - \gamma_{ij} \leq 0 \quad\quad\quad \forall\,i \in V, j \in N(i)  \label{eq:dual2}\\
 & \alpha_i,\lambda_i,\tau_i \geq 0 \quad\quad\quad\quad\quad\;\, \forall\, i \in V  \label{eq:dual3}\\
 & \beta_{ij},\gamma_{ij} \geq 0 \quad\quad\quad\quad\quad\quad \forall\,i \in V, j \in N(i)  \label{eq:dual4}
 \end{align}
 \end{subequations}} 

For any value of $w$, problem~\eqref{eq:LP} \textit{(i)} admits at least one feasible solution, 
\textit{(ii)} is bounded because both variables $x$ and $u$ are bounded. Thus, strong duality holds between problem~\eqref{eq:LP} (the LP relaxation of $\Psi$) and its dual~\eqref{eq:dualref_pre}. 
Given what we discussed before, in \eqref{eq:formulation_bil2}, we can replace $\Psi(w)$ by its linear relaxation~\eqref{eq:LP} since their optimal values are the same as proved in Theorem~\ref{th:integr}, and then replace problem~\eqref{eq:LP} by \eqref{eq:dualref_pre}, since their optimal values are the same by strong duality.
We can further drop the maximum operator, obtaining the following single-level formulation:
\begin{subequations}\label{eq:dualref}
	\begin{align}
	\min_{v,w,\alpha,\beta,\gamma,\lambda,\tau} & n - v\\
	\text{s.t.} &\; v \leq \;\sum_{i \in V}\left[(k-|N(i)|)\alpha_i +w_i\lambda_i-\tau_i\right] \label{eq:nonlin}\\
	& \sum\limits_{i \in V} w_i = b \\
 & \eqref{eq:dual1}-\eqref{eq:dual4}\\
 & v \in \mathbb{Z}, \;w \in \{0,1\}^n. 
 \end{align}
\end{subequations} 
This single-level formulation is a Mixed-Integer Nonlinear Programming (MINLP) problem, that has 
bilinear terms in \eqref{eq:nonlin}, given by $\sum_{i \in V} w_i\lambda_i$. One could linearize these bilinear terms using again McCormick reformulation, and/or applying other specialized techniques.
However, most of the these state-of-the-art techniques are integrated in modern MINLP solvers, thus we decided to hand over the compact model~\eqref{eq:dualref} as it is to the solver used in the experiments (see Section~\ref{sec:numerical-experiments} for details).
		
\section{Valid inequalities}\label{sec:valid_ineq}
In this section, we describe different classes of valid inequalities, which are used to strengthen the single-level formulations presented above. \revised{Some of them are valid for all the feasible solutions, while others cut off parts of the feasible domain due to symmetries or dominance conditions}. We point out that an initial \emph{pre-processing} procedure is applied to $\mathcal{G}$ which consists of removing all nodes not belonging to its $k$-core.

\subsection{Dominance and symmetry breaking inequalities}\label{subsec:preprocessing}
For any node $u\in V$, we can compute the $k$-core of $\mathcal{G} \setminus \{u\}$ and define $J_u$ as the set of nodes, including $u$ itself, which leave the graph when node $u$ is removed (say, the \textit{followers} of $u$, not to be confused with the follower agent solving the lower level):
\[J_u = \{v\in V : v\notin C_k(\mathcal{G}\backslash\{u\})\}.\]
In the example graph in Figure~\ref{fig:starting_graph}, as shown in Figures~\ref{fig:example_1} and \ref{fig:example_2}, with $k=3$, $J_{13}=\{13\},$ while $J_{1}=\{1, 2, 3, 4, 5, 8, 9\}$.
Using these sets, defined for each node in the graph, we can add dominance inequalities to our formulations. 
In the same example as before, $J_3=\{2, 3, 4, 5\}$ and we can observe that $J_3 \subset J_1$, meaning that every node that leave the network when node 3 is removed, would also leave the network when node 1 is removed.
In general, if $J_i \subset J_j$, i.e., the set of followers of node $i$ is strictly contained in the set of followers of node $j$, node $j$ should be removed first.
This can be imposed adding to the time-dependent formulation~\eqref{eq:formulation_td} the following inequalities:
\begin{equation}
 \label{eq:dominance-td}
 a_j^0 \le a_i^0 \qquad \forall\, i,j \in V : J_i \subset J_j
\end{equation}
which corresponds to adding the following inequalities to formulations~\eqref{eq:singlewithcuts}, and \eqref{eq:dualref}:
\begin{equation}
 \label{eq:dominance-bl}
 w_j \ge w_i\qquad \forall\, i,j \in V : J_i \subset J_j. 
\end{equation}
An additional family of valid inequalities is related to breaking symmetries among nodes which have the same set of followers. Let $S=\{i_1,\dots,i_{|S|}\}$ be an inclusion-wise maximal subset of nodes such that $J_{i_r} = J_{i_s}$ for all $i_r, i_s \in S$, i.e., $S$ contains all the nodes of the graph having a given set of followers, and there exist no superset of $S$ the nodes of which have the same set of followers. For example, in the graph in Figure~\ref{fig:starting_graph}, when $k=3$, nodes $4$ and $5$ have the same set of followers which is $\{4,5\}$ (if $4$ leaves the network, $5$ leaves it too and vice versa), thus a possible set $S$ is $\{4,5\}$.
Let the indices $i_r$, with $r\in\{1,\dots,|S|\}$, be given in increasing order.
Then we can break the symmetries by imposing that the node with the lowest index is removed first, i.e., 
\begin{equation}
 \label{eq:symmetry-td}
 a_{i_1}^0 \le \dots \le a_{i_{|S|}}^0 \qquad \forall\, S \subset V : J_{i_r} = J_{i_s}, \forall\, i_r, i_s \in S
\end{equation}
for the time-dependent formulation~\eqref{eq:formulation_td}, and 
\begin{equation}
 \label{eq:symmetry-bl}
 w_{i_1} \ge \dots \ge w_{i_{|S|}} \qquad \forall\, S \subset V : J_{i_r} = J_{i_s}, \forall\, i_r, i_s \in S
\end{equation}
for the single-level formulations~\eqref{eq:singlewithcuts} and \eqref{eq:dualref}.

\subsection{Valid inequalities to consider the cascade effect}
\label{subsec:cascade}
We present in this section a family of valid inequalities, related to the nodes which leave the network as a consequence of a single node or a set of nodes leaving. Such inequalities can be added to both the time-dependent formulation~\eqref{eq:formulation_td} and the leader's problem of the two bilevel formulations.

For any node $u\in V$, we can compute the $k$-core of $\mathcal{G} \setminus \{u\}$ and the set of followers $J_u$.
Given that by removing $u$, all nodes in $J_u$ will be removed as well, we can add a valid inequality stating that at most one node should be removed from $J_u$, that is:
\begin{equation}\label{eq:nofollowers-td}
\sum_{j\in J_u}a_j^0\geq |J_u|-1 \qquad \forall\, u \in V,
\end{equation}
for the time-dependent formulation, and 
\begin{equation}\label{eq:nofollowers-bl}
\sum_{j\in J_u}w_j \leq 1 \qquad \forall\, u \in V,
\end{equation}
for the formulations~\eqref{eq:singlewithcuts}, and \eqref{eq:dualref}.

\revised{Indeed, if there exists a feasible solution such that $u$ is not a collapser, but some of its followers are, i.e., $w_u = 0$ and $w_j =1$ for some $j \in J_u\setminus\{u\}$, then there exists an alternative solution with an objective value which is at least as good as the one of the former solution, where $u$ is the only collapser in $J_u$. Such a solution can be obtained by replacing all collapsed nodes $j$ by $u$, resulting in at least the same number of leaving nodes.}

We assume that the removal of less than $b$ nodes (together with the related followers) is not enough to empty the network. 
Under this assumption, Ineqs.~\eqref{eq:budget_td}, \eqref{eq:budget_bil}, \eqref{eq:budgetbil2}, requiring that exactly $b$ nodes are removed, as well as Ineqs.~\eqref{eq:bigM_total} and \eqref{eq:nogoodcut_total}, remain still valid when introducing constraints~\eqref{eq:nofollowers-td} and \eqref{eq:nofollowers-bl}. \revised{If instead this condition is not satisfied, then  inequalities~\eqref{eq:nofollowers-td} and \eqref{eq:nofollowers-bl} are no longer valid. We note that in social networks applications, this assumption is typically satisfied as $b<<n$.}

Inequalities~\eqref{eq:nofollowers-td} and \eqref{eq:nofollowers-bl} can be generalized to the case in which more than a single node is removed. Let $S \subseteq V$, with $|S| < b$, be such set of nodes. Assume
\begin{displaymath}
|C_k(\mathcal{G}\backslash S)| \ge b - |S|
\end{displaymath} holds, i.e., in the remaining $k$-core there are enough nodes to remove according to the budget left. The set $J_S$ of followers of $S$ (including $S$) can be defined as follows
\[
J_S = \{j\in V : j\notin C_k(\mathcal{G}\backslash S) \}
\]
and the inequality~\eqref{eq:nofollowers-td} is generalized into:
\begin{equation}
\label{eq:generalnofollowers-td}
\sum_{j\in J_S} a_j^0 \geq |J_S| - |S| \qquad \forall\, S \subseteq V: |S| < b, 
\end{equation}
while the inequality~\eqref{eq:nofollowers-bl} into:
\begin{equation}
\label{eq:generalnofollowers-bl}
\sum_{j\in J_S} w_j \leq |S| \qquad \forall\, S \subseteq V: |S| < b.
\end{equation}

The number of constraints~\eqref{eq:nofollowers-td} and \eqref{eq:nofollowers-bl} is equal to the number of nodes in the graph. For each node $u \in V$, the set $J_u$, i.e., the followers of $u$, can be easily obtained by computing the $k$-core of $\mathcal{G} \setminus \{u\}$ and the related constraint can be added to the model.
Instead, the number of constraints~\eqref{eq:generalnofollowers-td} and \eqref{eq:generalnofollowers-bl} is $\binom{n}{b-1}$.
Thus a separation routine is required.

\subsection{Lower bound on the solution value}\label{subsec:lowerbound}
Let $\{L_i\}_{i \in \{k,\dots,\ell\}}$ be a series of layers, each one containing the nodes of $\mathcal{G}$ with coreness equal to $i$, with $i$ being at least $k$ and at most $\ell$, where $\ell$ is the maximum coreness of any node in the graph.
For instance, in the graph in Figure~\ref{fig:example_coreness}, with $k=1$, we have $\ell = 3$, $L_1$ contains the cyan nodes, $L_2$ contains the green nodes, while $L_3$ contains the orange nodes. 

Let us consider the $h$-core $\bigcup_{i \in \{h,\dots,\ell\}} L_i$, where $h = k + b$. Even by removing any subset of $b$ nodes from such $h$-core, the remaining nodes still constitute a $k$-core. The size of the remaining $k$-core is a valid lower bound to the solution value of the Collapsed $k$-Core Problem.
Let us denote as $m := \left|\bigcup\nolimits_{i \in \{h,\dots,\ell\}} L_i \right| - b$ this lower bound. 
This means that we can restrict the set $\mathcal{Y}$ of the incident vectors of all the $k$-subcores of the graph (over which we optimize the lower-level problem) as follows:
\begin{equation}
		\mathcal{\tilde{Y}} = \left\{ y \in \{0,1\}^n \in \mathcal{Y}: \sum_{j \in \bigcup\nolimits_{i \in \{h,\dots,\ell\}} L_i} (1-y_j) \le b \right\}.
\end{equation}
Indeed the number of nodes which are not in the feasible $k$-cores belonging to the layers $\bigcup_{i \in \{h,\dots,\ell\}} L_i$ will not be greater than the budget $b$.

This corresponds to adding the following constraint to the time-dependent model~\eqref{eq:formulation_td}
 \begin{equation}\label{eq:lower_bound-td}
  \sum\limits_{i \in V} a_i^T \ge m.
 \end{equation}
Furthermore, a tighter upper bound on the number of deletion rounds can be defined as $T = n-b-m$. 

Similarly constraint
\begin{equation}\label{eq:lower_bound-bl1}
  z \ge m, 
 \end{equation}
can be added to formulation~\eqref{eq:singlewithcuts}, and \begin{equation}\label{eq:lower_bound-bl2}
  v \le n- m, 
 \end{equation}
to model~\eqref{eq:dualref}.
 
According to the defined lower bound, in a similar fashion as it is done in stochastic integer programming (see, e.g. \citet{LaporteL93}), we can also tighten the constraints of the sparse formulation~\eqref{eq:singlewithcuts},
presented in Section~\ref{subsubsec:sparse}.
Inequalities~\eqref{eq:bigM_total} can be restated as follows:
\begin{equation}\label{eq:bigMcut}
 z \geq m + {(|K| - m) } \Bigg[ 1 - \sum_{i \in 
 K} w_i \Bigg] \qquad \forall\,K \in \mathcal{K}: |K| > m 
\end{equation}

and inequalities \eqref{eq:nogoodcut_total} as follows:
\begin{equation}\label{eq:nogoodcut}
	z \ge m + (\left|C_k(\mathcal{G}\backslash W) \right| - m) \Bigg[ \sum_{i \in W} w_i - b + 1 \Bigg] \qquad \forall\, W \in \Omega.	\end{equation}

\subsection{Valid inequalities derived from $k$-subcores}\label{subsec:k-core}
For a given $K \in \mathcal{K}$, assume that $\delta(\mathcal{G}[K]) \geq k+1$ (the degree of nodes in the subgraph of $\mathcal{G}$ induced by $K$ is at least $k+1$). Assume we are given an interdiction policy $\tilde W \subset \Omega$ such that 
at most one of the nodes in $K$ is 
interdicted, then we have that $|C_k(\mathcal{G}\backslash\tilde W)| \geq |K|-1.$
Thus we can impose:
\begin{equation}\label{eq:hcores_bil}
 z \ge m + (|K|-1-m) \left[1-\sum_{i 
 \in K}\frac{w_i}{2}\right]
 \qquad \forall\, K \in \mathcal{K} : 
 \delta(\mathcal{G}[K]) \geq k+1
\end{equation}
in formulation~\eqref{eq:singlewithcuts}.

This can be easily generalized to the case in which $\delta(\mathcal{G}[K]) = h \geq k+1$ as follows:
\begin{equation}
 \label{eq:generalized_hcores_bil}
 z \ge m + (|K|-h+k-m) \left[ 1- \sum_{i \in K}\frac{w_i}{h-k+1}\right] \quad \forall\, K \in \mathcal{K} : 
 \delta(\mathcal{G}[K]) = h \geq k+1.
\end{equation} 
\revised{This means that, if at most $h-k$ nodes are removed from $K$, then the objective function value $z$ is lower bounded by $|K|-h+k$. Indeed, the nodes of the $h$-subcore which remain in the network will still have more than $k$ neighbours.}
\section{Separation procedures}\label{sec:separation-procedures}
The inequalities~\eqref{eq:dominance-td}, \eqref{eq:dominance-bl}, \eqref{eq:symmetry-td}, and \eqref{eq:symmetry-bl} introduced in Subsection~\ref{subsec:preprocessing}, as well as the inequalities~\eqref{eq:nofollowers-td} and \eqref{eq:nofollowers-bl} modeling the cascade effect following the leaving of a single node, introduced in Subsection~\ref{subsec:cascade}, and the ones related to the combinatorial lower bound $m$, i.e., \eqref{eq:lower_bound-td}, \eqref{eq:lower_bound-bl1} and \eqref{eq:lower_bound-bl2}, introduced in Subsection~\ref{subsec:lowerbound}, are added to the corresponding models during the initialization phase as they are in polynomial number. 
Specifically, we add the following $|V|+1$ inequalities: the $|V|$ inequalities~\eqref{eq:nofollowers-td}, or \eqref{eq:nofollowers-bl}, and the inequality \eqref{eq:lower_bound-td}, or \eqref{eq:lower_bound-bl1}, or \eqref{eq:lower_bound-bl2} (just one inequality for each model). As regards inequalities~\eqref{eq:dominance-td}, \eqref{eq:dominance-bl}, \eqref{eq:symmetry-td}, and \eqref{eq:symmetry-bl}, we add them from the beginning following an heuristic procedure here described. After computing the set of followers $J_j$ for each node $j \in V$, a dominance inequality of type~\eqref{eq:dominance-td} or~\eqref{eq:dominance-bl} is added to the corresponding formulation for each node $i \in J_j$ such that $J_i \subset J_j$.
Similarly, a partitioning $\mathcal{P}$ of the set of nodes $V$ into at most $|V|$ disjoint subsets is constructed, by iteratively assigning each node $i \in V$ to the subset which contains nodes having exactly the same followers of node $i$. Formally, for any $S$ in $\mathcal{P},$ $J_i = J_j, \forall\; i,j \in S$ and $J_i \neq J_j, \forall\; i \in S, \forall\; j \notin S$.
Hence, a symmetry breaking inequality of type~\eqref{eq:symmetry-td} or~\eqref{eq:symmetry-bl} is added to the appropriate formulation for each set $S$ of the partition $\mathcal{P}$.

Instead, the other valid inequalities introduced in Section~\ref{sec:valid_ineq} need a procedure to be separated.
In this section, we first present the separation procedure associated with compact formulations \eqref{eq:formulation_td} and \eqref{eq:dualref} and used to separate cuts~\eqref{eq:generalnofollowers-td} and~\eqref{eq:generalnofollowers-bl}. We then describe the separation procedures for the non-compact formulation~\eqref{eq:singlewithcuts} used to separate constraints~\eqref{eq:bigMcut}, \eqref{eq:nogoodcut} and inequalities~\eqref{eq:generalnofollowers-bl}, and \eqref{eq:generalized_hcores_bil}. We note that separation is made on integer solutions only \revised{by using the specific \textit{lazyconstraints} separation procedure provided by the commercial solver used in the experiments (Gurobi, in our case). Note that we decided to separate these inequalities on integer solutions only in order to avoid too many calls of the separation procedure (on fractional solutions) and speed-up the solution process}.

\subsection{Separation procedures for the compact formulations}\label{subsec:separation_compact}
A heuristic procedure for detecting violated inequalities~\eqref{eq:generalnofollowers-bl} added to formulation~\eqref{eq:dualref} is here described.

\begin{enumerate}
 \item Consider an interdiction policy $\bar{w}$ of the leader and let $\bar W$ denote the related set of collapsers;
 \item For each node $j \in \bar W$ do the following:
  \begin{enumerate}
   \item Compute the $k$-core of $\mathcal{G} \backslash \bar W \cup \{j\}$, i.e., $C_k(\mathcal{G} \backslash \bar W \cup \{j\})$;
   \item If $j \notin C_k(\mathcal{G} \backslash \bar W \cup \{j\})$, then $j \in J_{ \bar W \setminus \{j\}}$ and add a violated inequality of type~\eqref{eq:generalnofollowers-bl} with $S= \bar W \setminus \{j\}$ to formulation \eqref{eq:dualref}.
 \end{enumerate}
\end{enumerate}
Given a set $\bar W$ of collapsers, this routine is able to identify violated constraints of type~\eqref{eq:generalnofollowers-bl} where $S$ is given by the set of collapsers $\bar W$ excluding exactly one of them.

An analogous separation procedure is used to separate constraints~\eqref{eq:generalnofollowers-td} for formulation~\eqref{eq:formulation_td}, where we consider variables $\bar{a}^0$ and the related set $\bar A$ of collapsers instead of $\bar{w}$ and set $\bar W$ used in the procedure above.

\subsection{Separation procedures for the single-level formulation~\eqref{eq:singlewithcuts}}\label{subsec:separation_bil}
In the following, we explain how to heuristically separate constraints~\eqref{eq:generalnofollowers-bl}, \eqref{eq:bigMcut}, \eqref{eq:nogoodcut} and \eqref{eq:generalized_hcores_bil}. 
The procedure to separate~\eqref{eq:generalnofollowers-bl} is the same as the one described above. We repeat it here for the ease of reading.

We start initializing the relaxation of problem~\eqref{eq:singlewithcuts}, obtained by dropping constraints~\eqref{eq:bigM_total} and \eqref{eq:nogoodcut_total}.
Then, every time a feasible integer solution $\hat w$ of such relaxation is found, we compute the corresponding $C_k(\mathcal{G}\backslash \hat W)$ and:
  \begin{enumerate}
  \item \label{step_followers} For each node $j \in \hat W$ do the following:
    \begin{enumerate}
      \item Compute the $k$-core of
       $\mathcal{G} \backslash \hat W \cup \{j\}$, i.e., $C_k(\mathcal{G} \backslash \hat W \cup \{j\})$;
      \item If $j \notin C_k(\mathcal{G} \backslash \hat W \cup \{j\})$, i.e., $j \in J_{ \hat W \backslash \{j\}}$, add a violated inequality of type~\eqref{eq:generalnofollowers-bl} with $S= \hat W \backslash \{j\}$.
    \end{enumerate}
  \item \label{step_check} If no violated inequality of type~\eqref{eq:generalnofollowers-bl} is identified, go to step~\ref{step_bigM}, otherwise go to step~\ref{step_U}.
  \item \label{step_bigM} If $z < |C_k(\mathcal{G}\backslash\hat W)|$, add to the current relaxation a cut of the family~\eqref{eq:bigMcut} with $K=C_k(\mathcal{G}\backslash\hat W)$ and a cut of the family~\eqref{eq:nogoodcut} with $W = \hat W.$ 
  \item \label{step_U} Set $U = \emptyset$, and iteratively perform the following steps:
  \begin{enumerate}
    \item Select a node $u \in V\setminus U$ to remove and set $U := U \cup \{u\}$;
    \item Compute the $k$-core of $C_k(\{\mathcal{G}\backslash \hat W\} \setminus U)$;
    \item If $|C_k(\{\mathcal{G}\backslash \hat W\} \setminus U)| > m$, and $z < |C_k(\{\mathcal{G}\backslash \hat W\} \setminus U)|$ add a cut of the family~\eqref{eq:bigMcut} with $K = C_k(\{\mathcal{G}\backslash \hat W\} \setminus U)$. Otherwise, go to step 5. 
    \item If $|U|$ is over a given threshold, go to step 5.
    \end{enumerate}  
  \item\label{step_h}For $h \in \{k+1,\dots,\ell\}$: 
  \begin{enumerate}
    \item consider the set $K=\bigcup_{h \le i \le \ell} L_i$ of nodes of $C_k(\mathcal{G}\backslash \hat W)$ having coreness at least $h$.
    \item If $|K| \ge m$, then add cut~\eqref{eq:generalized_hcores_bil}. 
    \end{enumerate}
  \end{enumerate}
At step~\ref{step_followers} of the above presented procedure, we check if the collapsers are all \textit{really useful}. Indeed, we verify whether each $j \in \hat W$ is a follower of the other nodes in $\hat W$; if this is the case, removing $j$ is not useful for the leader: it will anyway disappear as follower of the other collapsers.

At step~\ref{step_check}, we verify if it is needed to perform step~\ref{step_bigM}. Indeed, if at least one of the inequalities~\eqref{eq:generalnofollowers-bl} has been added to the relaxation, there is no need to cut off the current solution by means of \eqref{eq:bigMcut} and \eqref{eq:nogoodcut}, being this solution already excluded by adding constraints of type~\eqref{eq:generalnofollowers-bl}.

At step~\ref{step_U}, we add a certain number (at most $|U|$) of Bender's like cuts of type~\eqref{eq:bigMcut} with $K$ of increasingly smaller dimension.
In order to obtain diversified sets of valid inequalities, nodes in $V \setminus U$ are selected in each iteration of step~\ref{step_U} according to decreasing order of the number of previously added constraints in which they are involved. This heuristic selection procedure means that the more a node is \textit{involved} in the previous steps, the less it will be considered in step~\ref{step_U}. 

At step~\ref{step_h}, at most $\ell-k$ inequalities of type~\eqref{eq:generalized_hcores_bil} are added. In particular, for any given $h \in \{\revised{k+}1,\dots,\ell\}$, the set of nodes in $C_k(\mathcal{G}\backslash \hat W)$ having core number at least $h$ is computed and used as the set $K$ in \eqref{eq:generalized_hcores_bil}.

\section{Numerical experiments}\label{sec:numerical-experiments}
In this section, we analyse the computational performance of the following four exact approaches:
\begin{itemize}
  \item Time-Dependent Model, corresponding to the compact ILP formulation~\eqref{eq:formulation_td}, see Section~\ref{subsec:timedep}.
  \item Nonlinear Model, corresponding to the compact MINLP formulation~\eqref{eq:dualref} presented in Section~\ref{subsubsec:dual}.
  \item Sparse Model, corresponding to the non-compact formulation~\eqref{eq:singlewithcuts} (cf.\ Section~\ref{subsubsec:sparse}) for which we implemented a Branch\&Cut (B\&C) approach.
  \item Bilevel Solver, corresponding to the bilevel formulation \eqref{eq:formulation_bil1} which is solved using a general purpose intersection-cut based solver for Mixed-Integer Bilevel Linear Problems proposed in \citep{ljubic2017}. 
\end{itemize}

On the one hand, the two compact formulations~\eqref{eq:formulation_td} and \eqref{eq:dualref} are solved using a state-of-the-art MINLP solver together with the separation procedures proposed in Subsection~\ref{subsec:separation_compact}. On the other hand, formulation~\eqref{eq:singlewithcuts} is solved using a Branch\&Cut method which iteratively builds the feasible set of the original bilevel formulation~\eqref{eq:formulation_bil1}, by adding cuts of type \eqref{eq:bigM_total}, \eqref{eq:nogoodcut_total} as well as separating the inequalities proposed in Section~\ref{sec:valid_ineq} through the separation procedure illustrated in Subsection~\ref{subsec:separation_bil}. 
The bilevel formulation~\eqref{eq:formulation_bil1} is instead solved as it is through the general purpose algorithm proposed in \citep{ljubic2017}.

The proposed formulations were implemented in Python 3.8 
and solved by using the Gurobi solver (version 9.5.2). The bilevel solver of \citet{ljubic2017} uses Cplex 12.7. 
The separation procedures presented in the paper are implemented within lazy callbacks, with the threshold on $|U|$ used in step 4d of separation procedure presented in Subsection~\ref{subsec:separation_bil} set to 10, and $\ell$ set to $k+2$.

All the experiments were conducted in single-thread mode, on a 2.3 GHz Intel Xeon E5 CPU, 128 GB RAM. A time limit of two hours of computation and a memory limit of 10 GB were imposed for every run.

\subsection{Benchmark Instances}\label{subsec:benchmarkinstances}
In order to test and compare the performances of the discussed methods, a set of {136} instances was arranged starting from {14} different networks collected from the literature. All the instances files are collected in the online public repository 
\url{https://bit.ly/collapsed-k-core}.

For each network, several combinations of values for $k$ and $b$ were selected by analysing the core number distribution of the nodes.
Table~\ref{table:instances_details} reports, for each network, the bibliographic source from which it was collected, the number of its nodes, the number of its edges, the different selected values for $k$,
as well as the associated sizes of the network (nodes and edges) after pre-processing and, finally, the selected values for the budget $b$.

\begin{table}[h!]
\centering
\resizebox*{!}{1\textwidth}{
\begin{tabular}{@{}ccccccc@{}}
\toprule \rowcolor[HTML]{EFEFEF} 
\textbf{network}    & \textbf{\#nodes}    & \textbf{\#edges}    & \textbf{k} & \textbf{\#nodes after pre-processing} & \textbf{\#edges after pre-processing} & \textbf{budget}\\ \midrule
\multirow{4}{*}{adjnoun\;\citep{Newman2006}}& \multirow{4}{*}{112}          & \multirow{4}{*}{425}           & \multicolumn{1}{c}{5}     & 63      & 298      & \{3\}   \\
{}     & {}  & {}  & 4    & 79      & 359      & \{3, 4, 5\}  \\
{}     & {}  & {}  & 3    & 89      & 389      & \{3, 4, 5\}  \\
{}     & {}  & {}  & 2    & 102      & 415      & \{3, 4, 5\}  \\ \midrule
\multirow{3}{*}{as-22july06 \citep{RouteViewsArchive}} & \multirow{3}{*}{22963} & \multirow{3}{*}{48436} & \multicolumn{1}{c}{15}     & 168      & 3115     & \{3, 4, 5\}   \\
& &  & 10    & 322      & 4845     & \{3, 4, 5\}  \\
& &  & 5    & 1087     & 9493     & \{3, 4, 5\}  \\ \midrule
\multirow{3}{*}{astro-ph \citep{Newman2001}}  & \multirow{3}{*}{16706} & \multirow{3}{*}{121251} & \multicolumn{1}{c}{42}     & 400      & 10552     & \{3, 4, 5\}   \\
& &  & 32    & 936      & 23433     & \{3, 4, 5\}  \\
& &  & 28    & 1393     & 32375     & \{3, 4, 5\}  \\ \midrule
\multirow{4}{*}{cond-mat \citep{Newman2001}}  & \multirow{4}{*}{16726} & \multirow{4}{*}{47594} & \multicolumn{1}{c}{9}     & 943      & 6573     & \{3, 4, 5\}   \\
& &  & \multicolumn{1}{c}{8}     & 1487     & 9544     & \{3, 4, 5\}   \\
& &  & \multicolumn{1}{c}{7}     & 2227     & 13280     & \{3, 4, 5\}   \\
& &  & \multicolumn{1}{c}{6}     & 3442     & 18713     & \{3, 4, 5\}   \\ \midrule
\multirow{4}{*}{cond-mat-2003 \citep{Newman2001} }        & \multirow{4}{*}{31163}         & \multirow{4}{*}{120029}         & \multicolumn{1}{c}{13}     & 1132     & 12732     & \{3, 4, 5\}   \\
{}     & {}  & {}  & 12    & 1609     & 17327     & \{3, 4, 5\}  \\
{}     & {}  & {}  & 10    & 2901     & 28339     & \{3, 4, 5\}  \\
{}     & {}  & {}  & 9    & 4071     & 36920     & \{3, 4, 5\}  \\ \midrule
\multirow{4}{*}{cond-mat-2005 \citep{Newman2001} }        & \multirow{4}{*}{40421  }       & \multirow{4}{*}{175692     }     & \multicolumn{1}{c}{14}     & 1793     & 24595     & \{3, 4, 5\}   \\
{}     & {}  & {}  & 13    & 2151     & 28640     & \{3, 4, 5\}  \\
{}     & {}  & {}  & 12    & 2808     & 35214     & \{3, 4, 5\}  \\
{}     & {}  & {}  & 11    & 3555     & 42346     & \{3, 4, 5\}  \\ \midrule
\multirow{3}{*}{dolphins \citep{Lusseau2003}}  & \multirow{3}{*}{62}  & \multirow{3}{*}{159}  & \multicolumn{1}{c}{4}     & 36      & 109      & \{3\}   \\
& &  & 3    & 45      & 135      & \{3, 4, 5\}   \\
& &  & 2    & 53      & 150      & \{3, 4, 5\}   \\ \midrule
\multirow{2}{*}{football \citep{Girvan2002}  }         & \multirow{2}{*}{115  }        & \multirow{2}{*}{613    }       & \multicolumn{1}{c}{8}     & 114      & 606      & \{3, 4, 5\}   \\
{}     & {}  & {}  & 7    & 115      & 613      & \{3, 4, 5\}  \\ \midrule
\multirow{4}{*}{hep-th \citep{Newman2001}}   & \multirow{4}{*}{8361} & \multirow{4}{*}{15751} & \multicolumn{1}{c}{7}     & 137      & 885      & \{3, 4, 5\}   \\
& &  & 6    & 358      & 1847     & \{3, 4, 5\}  \\
& &  & 5    & 851      & 3775     & \{3, 4, 5\}  \\
& &  & 4    & 1735     & 6552     & \{3, 4, 5\}  \\ \midrule
karate \citep{Zachary1977} & 34           & 78& 2     & 33      & 77      & \{3, 4, 5\}   \\ \midrule
\multirow{4}{*}{lesmis \citep{Knuth1993}  }          & \multirow{4}{*}{77 }          & \multirow{4}{*}{254  }         & \multicolumn{1}{c}{6}     & 38      & 186      & \{3, 4, 5\}   \\
{}     & {}  & {}  & 4    & 41      & 197      & \{3, 4, 5\}  \\
{}     & {}  & {}  & 3    & 48      & 215      & \{3, 4, 5\}  \\
{}     & {}  & {}  & 2    & 59      & 236      & \{3, 4, 5\}  \\ \midrule
\multirow{4}{*}{netscience \citep{Newman2006}} & \multirow{4}{*}{1589} & \multirow{4}{*}{2742}  & \multicolumn{1}{c}{5}     & 247      & 976      & \{3, 4, 5\}   \\
& &  & 4    & 470      & 1511     & \{3, 4, 5\}  \\
& &  & 3    & 751      & 2045     & \{3, 4, 5\}  \\
& &  & 2    & 1141     & 2535     & \{3, 4, 5\}  \\ \midrule
\multirow{4}{*}{polbooks \citep{Krebs}}  & \multirow{4}{*}{105}  & \multirow{4}{*}{441}  & \multicolumn{1}{c}{5}     & 65      & 300      & \{3, 4\}   \\
& &  & 4    & 98      & 422      & \{3, 4, 5\}  \\
& &  & 3    & 103      & 437      & \{3, 4, 5\}  \\
& &  & 2    & 105      & 441      & \{3, 4, 5\}  \\ \midrule
\multirow{3}{*}{power \citep{Watts1998}}    & \multirow{3}{*}{4941} & \multirow{3}{*}{6594}  & \multicolumn{1}{c}{4}     & 36      & 106      & \{3, 4, 5\}   \\
& &  & \multicolumn{1}{c}{3}     & 231      & 479      & \{3, 4, 5\}   \\
& &  & \multicolumn{1}{c}{2}     & 3353     & 5006     & \{3, 4, 5\}   \\ 
\bottomrule
\end{tabular}
}
\caption{Detailed description of the instance set.}
\label{table:instances_details}
\end{table}

\subsection{Effectiveness of the Collapsed $k$-Core formulations}
The detailed results obtained by testing the four methods on the instances described in Subsection~\ref{subsec:benchmarkinstances} are available online at \url{https://bit.ly/collapsed-k-core}.
In the following, we report summary tables and charts which we use to compare the tested methods and formulations. 
Because of the imposed memory limits, the Time-Dependent Model solves only for 87 instances while, for the remaining 48, even the relaxation at the root node is not solved. For this reason, we summarize in Table~\ref{tab:summary87} the results obtained by testing the four methods on this subset of 87 instances, reporting for each of them: \#opt, the number of optimal solutions found by the method within the limits; LB, the average lower bound; UB, the average upper bound; $\text{gap}_{\text{LB}}$[{\footnotesize\%}], the average percentage gap, where the gap is calculated as $\frac{100*(\text{UB}-\text{LB})}{\text{UB}}$ per each instance; time[s], the average computing time in seconds; B\&C nodes, the average number of nodes of the branch-and-cut tree at termination; $\text{LB}_\text{r}$, the average lower bound computed by solving the relaxation at the root node; $\text{gap}_{\text{UB}_\text{best}}$[{\footnotesize\%}], the average gap with respect to the best known solution (dimension of the $k$-core), calculated as $\frac{100*(\text{UB}-\text{UB}_{\text{best}})}{\text{UB}_{\text{best}}}$ per each instance; $\text{gap}_{\text{LB}_\text{r}}$[{\footnotesize\%}], the average percentage gap with respect to the root bound, calculated as $\frac{100*(\text{UB}-\text{LB}_\text{r})}{\text{UB}}$ per each instance.

\begin{table}[H]
\begin{adjustbox}{width=1\textwidth, margin=1ex 1ex 1ex 1ex}
\begin{tabular}{r|rrrrrrrrr}
\rowcolor[HTML]{EFEFEF} 
\multicolumn{1}{c|}{\cellcolor[HTML]{EFEFEF}}        & \textbf{\#opt} & \textbf{LB}         & \textbf{UB}         & \textbf{$\text{gap}_{\text{LB}}$[{\footnotesize\%}]}    & \textbf{time[s]}        & \textbf{B\&C nodes}         & \textbf{$\text{LB}_\text{r}$}       & \textbf{$\text{gap}_{\text{UB}_\text{best}}$[{\footnotesize\%}]} & \textbf{$\text{gap}_{\text{LB}_\text{r}}$[{\footnotesize\%}]} \vspace*{1mm}\\ \hline
\cellcolor[HTML]{EFEFEF}\textbf{Time-Dependent Model} & \textbf{29}  & 84.6 & 207.0          & 41.2           & 5234           & 79084 & 68.7          & 4.66 & 71.5       \\
\cellcolor[HTML]{EFEFEF}\textbf{Sparse Model} & \textbf{44}  & 102.5 & 207.4          & 30.5           & 4045          & 30508 & 81.6           & 1.91 & 70.8      \\
\cellcolor[HTML]{EFEFEF}\textbf{Nonlinear Model}   & \textbf{52}  & {112.3} & {201.7} & {20.4} & {3347} & {1110881} & {81.6} & 0.38 & 70.4 \\      
\cellcolor[HTML]{EFEFEF}\textbf{Bilevel Solver}       & \textbf{26}  & 23.1 & 212.1           & 55.8          & 5297           & 18263 & 2.67 & 4.68 & 96.5      
\end{tabular}
\end{adjustbox}
\caption{Summary of the computational performance of the four exact methods
on the subset of 87 instances solved by the time-dependent model.}
\label{tab:summary87}
\end{table}

The results show a clear superiority of the Nonlinear Model, both in terms of time and solution quality. Indeed, the Nonlinear Model provides the highest number of optimal solutions among the tested methods, yielding 52 out of 87 instances solved to optimality. Furthermore, the average computing time required by the Nonlinear Model is considerably lower than the one required by the other formulations; also, the provided average final lower and upper bounds values and gaps are tighter. Specifically, the lower bound $m$ computed in the preprocessing procedure is $81.6$ on average. The improvement with respect to this value is much higher for the Nonlinear Model than for the other approaches which take this lower bound into account (i.e., all the approaches, but the bilevel solver).
On average, the number of nodes explored by the branch-and-cut approach solving the Nonlinear Model is greater than the number of nodes explored by the one solving the Sparse Model and the number of nodes explored by the Bilevel Solver. This reflects the fact that the problems considered at each node of the branch-and-cut tree solving the Nonlinear Model are easier to solve with respect to the ones of the other models, so that, in the same amount of time, more nodes are explored.

All the three proposed problem-specific methods exhibit better performing behaviors than the general purpose Bilevel Solver.
\revised{To better visualize this computational dominance, three summary charts related to the four methods solving the considered 87 instances are reported in~\ref{appendix2}.}

Since the imposed memory limits prevented the Time-Dependent Model from solving the remaining instances, from now on we restrict the comparison to the other three methods and consider the whole instance set described in Subsection~\ref{subsec:benchmarkinstances}.
In particular, in Table~\ref{tab:summary136}, we report the same information as in Table~\ref{tab:summary87}, but this time for the whole set of 136 instances, with respect to the three following methods: Sparse Model, Nonlinear Model and Bilevel Solver.

\begin{table}[H]
\begin{adjustbox}{width=1\textwidth, margin=1ex 1ex 1ex 1ex}
\begin{tabular}{r|rrrrrrrrr}
\rowcolor[HTML]{EFEFEF} 
\multicolumn{1}{c|}{\cellcolor[HTML]{EFEFEF}}        & \textbf{\#opt} & \textbf{LB}         & \textbf{UB}         & \textbf{$\text{gap}_{\text{LB}}$[{\footnotesize\%}]}    & \textbf{time[s]}        & \textbf{B\&C nodes}         & \textbf{$\text{LB}_\text{r}$}       & \textbf{$\text{gap}_{\text{UB}_\text{best}}$[{\footnotesize\%}]} & \textbf{$\text{gap}_{\text{LB}_\text{r}}$[{\footnotesize\%}]} \vspace*{1mm}\\ \hline
\cellcolor[HTML]{EFEFEF}\textbf{Sparse Model}          & \textbf{44}  & 266.8          & 911.4           & 46.3           & 5182           & 21919 & 253.3          & 1.89 & 72.1      \\
\cellcolor[HTML]{EFEFEF}\textbf{Nonlinear Model} & \textbf{52}  & {274.0} & {892.8} & {39.6} & {4735} & {744905} & {253.3} & 0.04 & 71.6\\     
\cellcolor[HTML]{EFEFEF}\textbf{Bilevel Solver}     & \textbf{26}  & 22.4& 915.9           & 71.3          & 5984          & 11892 & 3.80 & 3.76        & 97.7     
\end{tabular}
\end{adjustbox}
\caption{Summary of the computational performances of Sparse Model, Nonlinear Model and Bilevel Solver on the whole set of 136 instances.}
\label{tab:summary136}
\end{table}

The results on the whole set of instances confirm the computational dominance of the Nonlinear Model, which solves 52 out of the 136 instances to optimality and almost always provides solution values which are better than or equal to the ones found by the other methods, with an average gap of 0.04\%, computed with respect to the best known feasible solutions. Again, the number of branch-and-cut nodes reflects the faster resolution of the continuous relaxation of the Nonlinear Model at branching nodes.

We further provide three summary charts related to the three methods solving all the 136 instances. The first chart, shown in Figure~\ref{fig:1}, reports the number of instances solved to optimality within a given computational time. 
The second one, in Figure~\ref{fig:2}, shows the optimality gap at termination, i.e., what we called $\text{gap}_{\text{LB}}$[{\footnotesize\%}]. In particular, the plot shows the number of instances (on the vertical axis) for which the gap at termination is smaller than or equal to the value reported on the horizontal axis.
Figure~\ref{fig:3} reports the gap between the feasible solution at termination, and the best found feasible collapsed $k$-core among the three compared approaches, i.e., what we called $\text{gap}_{\text{UB}_\text{best}}$[{\footnotesize\%}]. Again, the chart shows the number of instances (on the vertical axis) for which the value of $\text{gap}_{\text{UB}_\text{best}}$ is smaller than or equal to the value reported on the horizontal axis.

\begin{figure}[h!]
    \centering
    \includegraphics[width=13cm,height=6.5cm]{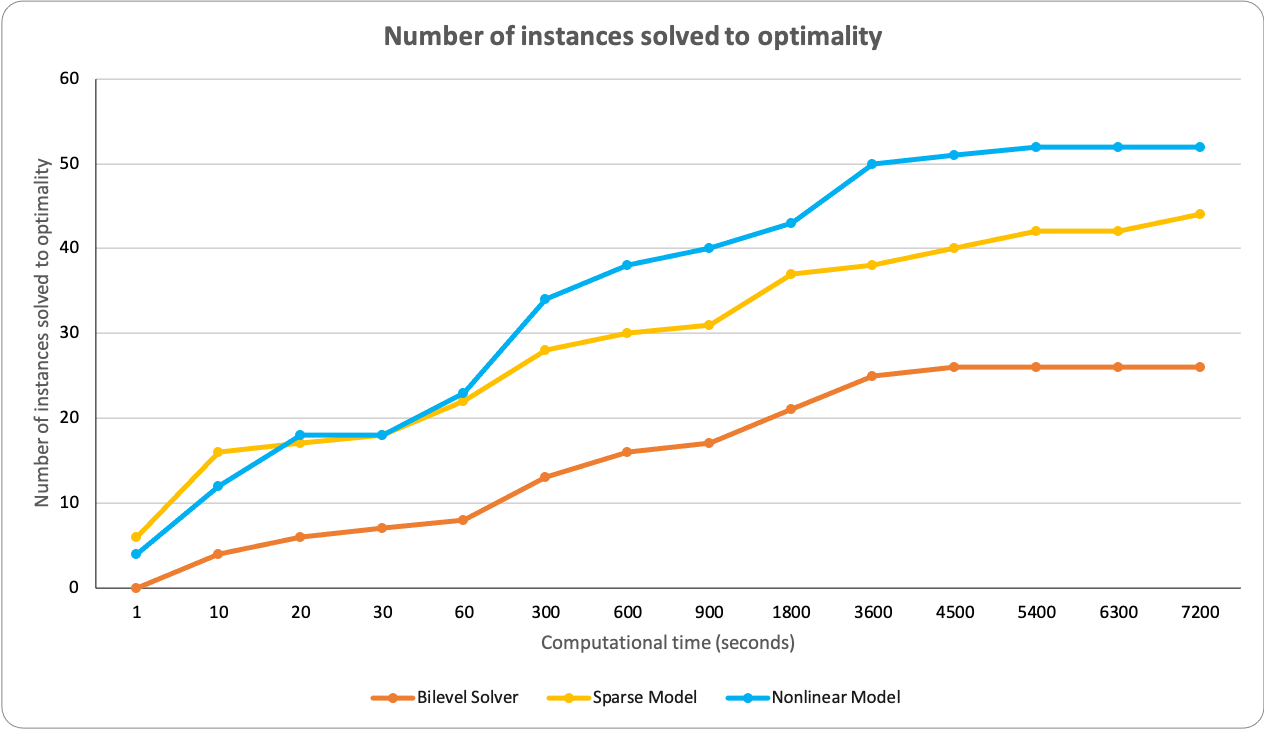}
    \caption{Cumulative chart of number of instances solved to optimality within a given computational time.}
    \label{fig:1}
\end{figure}

\begin{figure}[h!]
    \centering
    \subfloat[Cumulative chart of percentage gap at termination.]{\includegraphics[width=7cm,height=5.5cm]{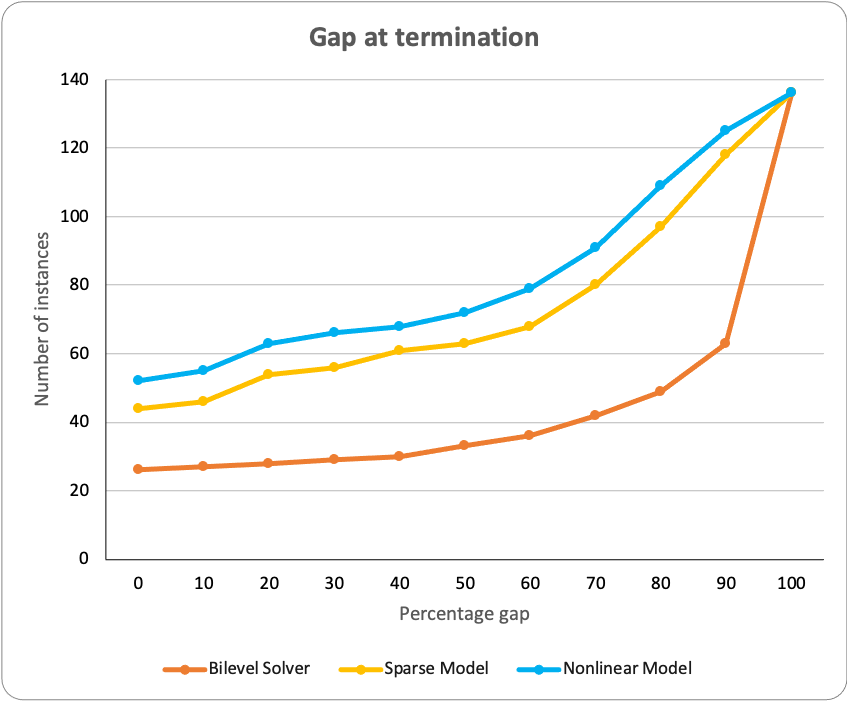}
    \label{fig:2}}\hfill
    \centering
    \subfloat[Cumulative chart of percentage gap with respect to the best feasible 
    solution at termination.]{\includegraphics[width=7cm,height=5.5cm]{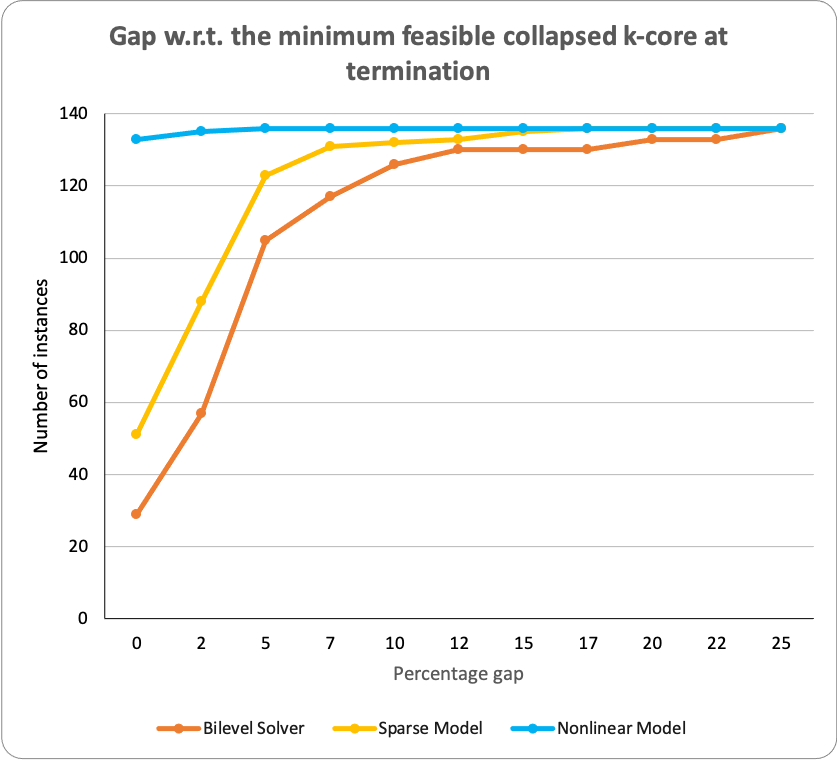}
    \label{fig:3}}
    \caption{Cumulative charts of two different percentage gaps.\vspace*{-3mm}}
\end{figure}
Overall, all the three charts show that the two approaches proposed in this paper are much more effective than the bilevel solver, which is largely outperformed by each of them. In the first chart (Figure~\ref{fig:1}), it can be observed that, only when the computational time is strictly below 20 seconds, the number of instances solved to optimality by the Sparse Model is slightly greater than the number of instances solved to optimality by the Nonlinear Model. However, when the considered time is larger than 20 seconds, the Nonlinear Model dominates the other two approaches. The chart in Figure~\ref{fig:2} demonstrates that the Nonlinear Model produces a gap at termination which is always lower than the one returned by the other approaches. Finally, Figure~\ref{fig:3} shows that, for more than 130 instances out of 136, the Nonlinear Model produces the best feasible solution, and, for the remaining 4 instances, the gap with respect to the best feasible solution found by one of the other models is below 3\%. Concerning the Sparse Model, it returns the best feasible solution for about 50 instances, and the gap for the remaining instances is less than 17\%. Finally, the bilevel solver finds the best feasible solution only for about 30 instances out of 136, and a gap with respect to the best one  can be as high as 25\%.

\revised{We also made an additional set of tests to verify the effectiveness of the valid inequalities presented in Section~\ref{sec:valid_ineq}. Specifically, we used model~\eqref{eq:singlewithcuts} as a benchmark and tested the following configurations:
\begin{inparaenum}[(i)]
    \item without any valid inequalities;
    \item with symmetry breaking inequalities \eqref{eq:dominance-bl}, \eqref{eq:symmetry-bl}, and lower bound constraint~\eqref{eq:lower_bound-bl1};
    \item with symmetry breaking inequalities \eqref{eq:dominance-bl}, \eqref{eq:symmetry-bl}, lower bound constraint~\eqref{eq:lower_bound-bl1}, followers' inequalities~\eqref{eq:nofollowers-bl} added to the formulation from the beginning, and constraints~\eqref{eq:generalnofollowers-bl} separated during the B\&C;
    \item with symmetry breaking inequalities \eqref{eq:dominance-bl}, \eqref{eq:symmetry-bl}, lower bound constraint~\eqref{eq:lower_bound-bl1}, and constraints~\eqref{eq:bigMcut} separated during the B\&C;
    \item with symmetry breaking inequalities \eqref{eq:dominance-bl}, \eqref{eq:symmetry-bl}, lower bound constraint~\eqref{eq:lower_bound-bl1}, and constraints~\eqref{eq:generalized_hcores_bil} separated during the B\&C;
    \item the Sparse Model with all the introduced valid inequalities (i.e., the model considered in the comparisons with the other models above).
    \end{inparaenum}
Results are summarized in Table~\ref{tab:InequalitiesPerformance} where we report, for each tested configuration: 
the number of optimal solutions found (\#opt), the average percentage gap at termination (gap[\%]), the average computing time (time[s]), the number of branch-and-cut nodes (B\&C nodes), and the number of added inequalities: \eqref{eq:dominance-bl}, \eqref{eq:symmetry-bl}, \eqref{eq:nofollowers-bl}, \eqref{eq:generalnofollowers-bl},  \eqref{eq:bigMcut}, and \eqref{eq:generalized_hcores_bil}, respectively.}

\hspace*{-1cm}\begin{table}[htbp]
  \centering
  \scalebox{0.55}{
    \begin{tabular}{r|rrrrrrrrrrr}
    \rowcolor[rgb]{ .937,  .937,  .937}       & \textbf{Model} & \textbf{\#opt} & \textbf{gap[\%]} & \textbf{time[s]} & \textbf{B\&C nodes} & \textbf{\#\eqref{eq:dominance-bl}} & \textbf{\#\eqref{eq:symmetry-bl}} & \textbf{\#\eqref{eq:nofollowers-bl}} & \textbf{\#\eqref{eq:generalnofollowers-bl}} & \textbf{\#\eqref{eq:bigMcut}} & \textbf{\#\eqref{eq:generalized_hcores_bil}} \\
    \midrule
    \cellcolor[rgb]{ .937,  .937,  .937} \textbf{(i)} & \textit{Model \eqref{eq:singlewithcuts}} & 24 & 82.4 & 6093 & 6388 & 0 & 0 & 0 & 0 & 0 & 0 \\
    \cellcolor[rgb]{ .937,  .937,  .937} \textbf{(ii)} & \textit{Model \eqref{eq:singlewithcuts} + \eqref{eq:dominance-bl} + \eqref{eq:symmetry-bl} + \eqref{eq:lower_bound-bl1}} & 27 & 56.2 & 5955 & 5356 & 1221 & 144 & 0 & 0 & 0 & 0 \\
    \cellcolor[rgb]{ .937,  .937,  .937} \textbf{(iii)} & \textit{Model \eqref{eq:singlewithcuts} + \eqref{eq:dominance-bl} + \eqref{eq:symmetry-bl} + \eqref{eq:lower_bound-bl1} + \eqref{eq:nofollowers-bl} + \eqref{eq:generalnofollowers-bl}} & 29 & 55.3 & 5847 & 6260 & 1221 & 144 & 303 & 242 & 0 & 0 \\
    \cellcolor[rgb]{ .937,  .937,  .937} \textbf{(iv)} & \textit{Model \eqref{eq:singlewithcuts} + \eqref{eq:dominance-bl} + \eqref{eq:symmetry-bl} + \eqref{eq:lower_bound-bl1} + \eqref{eq:bigMcut}} & 44 & 47.5 & 5132 & 16545 & 1221 & 144 & 0 & 0 & 131068 & 0 \\
    \cellcolor[rgb]{ .937,  .937,  .937} \textbf{(v)} & \textit{Model \eqref{eq:singlewithcuts} + \eqref{eq:dominance-bl} + \eqref{eq:symmetry-bl} + \eqref{eq:lower_bound-bl1} + \eqref{eq:generalized_hcores_bil}} & 29 & 53.0 & 5907 & 23465 & 1221 & 144 & 0 & 0 & 0 & 51895 \\
    \cellcolor[rgb]{ .937,  .937,  .937} \textbf{(vi)} & Sparse Model & 44 & 46.3 & 5182 & 21919 & 1221 & 144 & 303 & 183 & 91226 & 12631 \\
    \end{tabular}%
    }
    \caption{Performance of the Sparse model with several combinations of valid inequalities.}
  \label{tab:InequalitiesPerformance}%
\end{table}%

\revised{As expected, the number of optimal solutions found, as well as the average percentage gap at termination improve when adding further valid inequalities to the basic model~\eqref{eq:singlewithcuts}. We can also notice that the average time does not increase when increasing the number of considered valid inequalities. Inequalities~\eqref{eq:bigMcut} turn out to be the most effective, leading to 44 optimal solutions (and an average gap of 47.1\%), as many as the ones found by the complete Sparse Model (with an average gap of 46.3\%). The average number of added inequalities~\eqref{eq:dominance-bl}, \eqref{eq:symmetry-bl}, and \eqref{eq:nofollowers-bl} is constant for all configurations since they are not separated, but added to the model from the beginning, together with the lower bound constraint~\eqref{eq:lower_bound-bl1}.}

Finally, we provide two sensitivity analysis tables in which we group instances into three groups: \textit{Small} (with $n \le 100$), \textit{Medium} (with $100 <n \le1000$) and \textit{Large} (with $n>1000$), where $n$ is the number of nodes of the network after the pre-processing procedures. For each group, we consider three values of the budget $b \in \{3, 4, 5\}$, thus obtaining nine classes of instances. For each class, we report in Tables~\ref{tab:3} and \ref{tab:4} the total number of instances in the class (\#total).
In Table~\ref{tab:3}, for each method we report the number of instances solved to optimality (\#opt), the average computing time in seconds (time[s]) and the average number of nodes of the branch-and-cut tree (B\&C nodes); additionally, for the Sparse Model, we also report the average number of added cuts.
In Table~\ref{tab:4}, instead, we report for each class $\text{gap}_\text{LB}$[\%], $\text{gap}_{\text{LB}_\text{r}}$[\%], and $\text{gap}_{\text{UB}_\text{best}}$[\%].

The reported values show how increasing the value of the budget affects the different methods.
As expected, when the value of $b$ increases, the number of instances solved to optimality decreases for each method, while the average solution time and the average gap value increase.
Specifically, all the instances of the first class, i.e., small instances with $b=3$, are solved to optimality by both the Sparse and Nonlinear models, while the Bilevel Solver provides the optimal certified solution for all except two of them. The Nonlinear Model solves to optimality also all the instances of the second class, i.e., small instances with $b=4$. For the remaining classes, instead, no method is able to certify the optimality of all the instances from a given subclass. In particular, no optimal solution is found for any of the instances from the group ``large'', within the imposed time limit.
Overall, this analysis shows that the difficulty of an instance is highly affected by the budget value, in addition to the network size, but, for non-large instances, the Nonlinear Model confirms its superiority with respect to the remaining approaches.

\begin{table}[H]
\begin{adjustbox}{width=1\textwidth, margin=1ex 1ex 1ex 1ex}
  \begin{tabular}{rrr|rrrr|rrr|rrr}
  \rowcolor[HTML]{EFEFEF} 
  \multicolumn{3}{c|}{\cellcolor[HTML]{EFEFEF}\textbf{Class details}}      & \multicolumn{4}{c|}{\cellcolor[HTML]{EFEFEF}\textbf{Sparse Model}}         & \multicolumn{3}{c|}{\cellcolor[HTML]{EFEFEF}\textbf{Nonlinear Model}} & \multicolumn{3}{c}{\cellcolor[HTML]{EFEFEF}\textbf{Bilevel Solver}} \\ \hline
  \rowcolor[HTML]{EFEFEF} 
  \multicolumn{1}{c|}{\textbf{size}}       & \textbf{b} & \textbf{\#total} 
  & \textbf{\#opt} & \textbf{time[s]} & \textbf{B\&C nodes} & \textbf{cuts} & \textbf{\#opt} & \textbf{time[s]}    & \textbf{B\&C nodes}    & \textbf{\#opt}   & \textbf{time[s]}   & \textbf{B\&C nodes}  \\ \hline
  \multicolumn{1}{c|}{\cellcolor[HTML]{EFEFEF}}  & 3     & 14        
  & 14       & 102       & 5481     & 1719   & 14          & 27.6 & 26633        & 12         & 1623         & 10455      \\
  \multicolumn{1}{c|}{\cellcolor[HTML]{EFEFEF}}  & 4     & 12        
  & 11       & 1296       & 31205     & 5409   & 12          & 242 & 332382       & 7         & 3574         & 32115      \\
  \multicolumn{1}{c|}{\multirow{-3}{*}{\cellcolor[HTML]{EFEFEF}\textbf{Small}}} & 5     & 11        
  & 6       & 3820       & 53377     & 9226   & 10          & 1413           & 1960571        & 4         & 4805         & 45443      \\ \hline
  \multicolumn{1}{c|}{\cellcolor[HTML]{EFEFEF}}  & 3     & 17        
  & 9       & 4445       & 23730     & 16602      & 9           & 3961           & 353716       & 1         & 6884         & 11410 \\
  \multicolumn{1}{c|}{\cellcolor[HTML]{EFEFEF}}  & 4     & 17        
  & 3       & 6431       & 34573     & 19101   & 6           & 5544           & 1381271  & 1  & 6863     & 10923      \\
  \multicolumn{1}{c|}{\multirow{-3}{*}{\cellcolor[HTML]{EFEFEF}\textbf{Medium}}} & 5     & 17        
  & 1       & 6777       & 39951     & 16420      & 1           & 6938           & 2460756       & 1         & 6817         & 10669      \\ \hline
  \multicolumn{1}{c|}{\cellcolor[HTML]{EFEFEF}}  & 3     & 16        
  & 0       & 7200       & 5750     & 11795   & 0           & 7200           & 90644        & 0         & 7200         & 573       \\
  \multicolumn{1}{c|}{\cellcolor[HTML]{EFEFEF}}  & 4     & 16        
  & 0       & 7200       & 4393     & 9417     & 0           & 7200          & 80239        & 0         & 7200         & 492       \\
  \multicolumn{1}{c|}{\multirow{-3}{*}{\cellcolor[HTML]{EFEFEF}\textbf{Large}}} & 5     & 16        
  & 0       & 7200       & 6881     & 6206   & 0           & 7200           & 82350        & 0         & 7200         & 474       \\
  \end{tabular}
\end{adjustbox}
\caption{Sensitivity analysis showing the effect of different budget values on the number of instances solved to optimality, the solution time, the number of nodes and the number of cuts.}
\label{tab:3}
\end{table}

\begin{table}[H]
\begin{adjustbox}{width=1\textwidth, margin=1ex 1ex 1ex 1ex}
  \begin{tabular}{rrr|rrr|rrr|rrr}
  \rowcolor[HTML]{EFEFEF} 
  \multicolumn{3}{c|}{\cellcolor[HTML]{EFEFEF}\textbf{Class details}}         & \multicolumn{3}{c|}{\cellcolor[HTML]{EFEFEF}\textbf{Sparse Model}}       & \multicolumn{3}{c|}{\cellcolor[HTML]{EFEFEF}\textbf{Nonlinear Model}}         & \multicolumn{3}{c}{\cellcolor[HTML]{EFEFEF}\textbf{Bilevel Solver}}  \\ \hline
  \rowcolor[HTML]{EFEFEF} 
  \multicolumn{1}{c|}{\cellcolor[HTML]{EFEFEF}\textbf{size}}           & \textbf{b} & \textbf{\#Total} 
  & \textbf{$\text{gap}_\text{LB}$[{\footnotesize\%}]} & \textbf{$\text{gap}_{\text{LB}_\text{r}}$[{\footnotesize\%}]} & \textbf{$\text{gap}_{\text{UB}_\text{best}}$[{\footnotesize\%}]} & \textbf{$\text{gap}_\text{LB}$[{\footnotesize\%}]} & \textbf{$\text{gap}_{\text{LB}_\text{r}}$[{\footnotesize\%}]} & \textbf{$\text{gap}_{\text{UB}_\text{best}}$[{\footnotesize\%}]}  & \textbf{$\text{gap}_\text{LB}$[{\footnotesize\%}]} & \textbf{$\text{gap}_{\text{LB}_\text{r}}$[{\footnotesize\%}]} & \textbf{$\text{gap}_{\text{UB}_\text{best}}$[{\footnotesize\%}]}  \vspace*{1mm}\\ \hline
  \multicolumn{1}{c|}{\cellcolor[HTML]{EFEFEF}}      & 3     & 14        
  & 0.00          & 75.1          & 0.00       & 0.00          & 75.1         & 0.00        & 7.40         & 92.5          & 0.28 \\
  \multicolumn{1}{c|}{\cellcolor[HTML]{EFEFEF}}      & 4     & 12        
  & 6.40        & 82.4          & 0.11       & 0.00          & 82.4          & 0.00         & 20.2        & 95.6          & 1.28 \\
  \multicolumn{1}{c|}{\multirow{-3}{*}{\cellcolor[HTML]{EFEFEF}\textbf{Small}}} & 5     & 11        
  & 34.5         & 87.7          & 1.40     & 3.10        & 87.7         & 0.00        & 34.4        & 97.2          & 2.19 \\ \hline
  \multicolumn{1}{c|}{\cellcolor[HTML]{EFEFEF}}      & 3     & 17       
  & 22.4        & 57.3         & 0.46     & 22.6        & 57.1          & 0.17      & 78.0       & 97.0          & 5.35 \\
  \multicolumn{1}{c|}{\cellcolor[HTML]{EFEFEF}}      & 4     & 17       
  & 50.5        & 65.5          & 2.52       & 34.2        & 64.9          & 0.12       & 84.2        & 97.9          & 6.89 \\
  \multicolumn{1}{c|}{\multirow{-3}{*}{\cellcolor[HTML]{EFEFEF}\textbf{Medium}}} & 5     & 17        
  & 66.5        & 73.3          & 3.43      & 56.5         & 72.5          & 0.00        & 85.8        & 99.0          & 7.87 \\ \hline
  \multicolumn{1}{c|}{\cellcolor[HTML]{EFEFEF}}      & 3     & 16       
  & 64.0        & 64.0         & 2.14       & 63.0        & 63.1          & 0.00      & 98.8        & 99.7         & 2.06 \\
  \multicolumn{1}{c|}{\cellcolor[HTML]{EFEFEF}}      & 4     & 16        
  & 73.4        & 73.4          & 2.69       & 72.3        & 72.4          & 0.00        & 99.1        & 99.8          & 2.62 \\
  \multicolumn{1}{c|}{\multirow{-3}{*}{\cellcolor[HTML]{EFEFEF}\textbf{Large}}} & 5     & 16        
  & 79.8        & 79.8          & 3.38     & 78.9        & 78.9          & 0.00        & 99.5        & 99.9          & 3.22
  \end{tabular}
\end{adjustbox}
\caption{Sensitivity analysis showing the effect of different budget values on the gaps.}
\label{tab:4}
\end{table}

\section{Conclusion}\label{sec:conclusion}
Identifying the most critical users, in terms of network engagement, is a compelling topic in social networks analysis. Users who leave a community potentially affect the cardinality of its $k$-core, i.e., the maximal induced subgraph of the network with minimum degree at least $k$. In this paper, we presented different mathematical programming formulations of the Collapsed $k$-Core Problem, consisting in finding the $b$ nodes of a graph the removal of which leads to the $k$-core of minimal cardinality.
We started with a time-indexed compact formulation which models the cascade effect following the removal of the $b$ nodes. Then, we proposed two different bilevel programming models of the problem. In both of them, the leader aims to minimize the cardinality of the $k$-core obtained by removing exactly $b$ nodes. The follower wants to detect the $k$-core obtained after the decision of the leader on the $b$ nodes to remove, i.e., finding the maximal subgraph of the new graph where all the nodes have degree at least $k$. The two formulations differ in the way the follower's problem is modeled. In the first one, the lower level is an ILP model. It is solved through a Benders-like decomposition approach. In the second bilevel formulation, the lower level is modeled through LP, which we dualized in order to end up with a single-level formulation. Preprocessing procedures, and valid inequalities have been further introduced to enhance the proposed formulations.

In order to evaluate the proposed formulations we tested different existing instances, showing the superiority of the single-level reformulation of the second bilevel model. We further compared the approaches with the general purpose bilevel solver proposed in \citep{ljubic2017} which is outperformed by our problem-specific solution methods. \revised{The efficiency of the proposed valid inequalities is also demonstrated by performing additional computational experiments.}

Interesting directions for future research are related, for example, to applying our approaches to the $k$-core minimization problem through edge deletion \citep{zhu2018}. Furthermore, other related problems, like the so-called Anchored $k$-core Problem \citep{bhawalkar2015} may benefit from the proposed solution techniques. 

\bibliographystyle{elsarticle-harv}
\bibliography{collapsedkcore}
\newpage
\appendix
\revised{\section{Notation}\label{appendix1}
\revised{Table~\ref{tab:parameters} reports a list of mathematical notations used throughout the paper.}
\begin{table}[H]
    \centering
    \begin{adjustbox}{width=1\textwidth, margin=1ex 1ex 1ex 1ex}
    \begin{tabular}{r|l}
        \textbf{Notation} & \textbf{Definition} \\
        \hline
    \rowcolor[HTML]{EFEFEF}
    $V$ & set of nodes of graph $\mathcal{G}$\\
    $E$ & set of edges of graph $\mathcal{G}$\\
    \rowcolor[HTML]{EFEFEF}
    $N(i) := \{j\in V~:~\{i,j\}\in E\}$ & set of neighbors of node $i\in V$ \\
    $\delta(\mathcal{G}):=\min_{i \in V} |N(i)| $ & minimum degree of graph $\mathcal{G}$ \\
    \rowcolor[HTML]{EFEFEF}
    $\mathcal{G}[S]:=(S,E_S)~|~E_S=\{ \{i,j\}\in E:i,j\in S\}$ & subgraph of graph $\mathcal{G}$ induced by the set of nodes $S$\\
    $\mathcal{G}\backslash S := \mathcal{G}[V\backslash S]$ & subgraph of graph $\mathcal{G}$ induced by the set of nodes $V\backslash S$\\
    \rowcolor[HTML]{EFEFEF} 
    $\mathcal{Y}$ & set of incident vectors of any $k$-subcore of graph $\mathcal{G}$ (Eq.~\eqref{eq:kcore1Y}) \\
    $\mathcal{K} := \{K \subseteq V~:~\delta(\mathcal{G}[K]) \geq k\}$ & set of the $k$-subcores $K$ of graph $\mathcal{G}$\\
    \rowcolor[HTML]{EFEFEF} 
    $C_k(\mathcal{G}):=\arg\max_{K \in \mathcal{K}} |K|$ & $k$-core of graph $\mathcal{G}$\\
    $\mathcal{W}$ & set of incident vectors of any interdiction policy (Eq.~\eqref{eq:Wset}) \\
    \rowcolor[HTML]{EFEFEF} 
    $\Omega:=\{W\subseteq V ~:~|W| = b\}$ & set of the interdiction policies $W$\\
    $J_u := \{v\in V : v\notin C_k(\mathcal{G}\backslash\{u\})\}$ & set of the followers of node $u$, including $u$ itself\\
    \rowcolor[HTML]{EFEFEF} 
    $J_S := \{j\in V : j\notin C_k(\mathcal{G}\backslash S) \}$ & set of the followers of the nodes in the set $S$ including $S$ itself\\
    \end{tabular}
    \end{adjustbox}
    \caption{List of recurring mathematical notations.}
    \label{tab:parameters}    
\end{table}
}
\revised{\section{Cumulative charts on a set of 87 instances}\label{appendix2}

We report here three cumulative charts related to the four methods (Time Dependent Model, Sparse Model, Nonlinear Model, and Bilevel Solver) solving 87 instances. The first chart in Figure~\ref{fig:1a} reports the number of instances solved to optimality within a given computational time. The second one, in Figure~\ref{fig:2a}, shows the number of instances solved with a gap at termination which is smaller than or equal to the value reported on the horizontal axis. The last one, in Figure~\ref{fig:3a}, shows the number of instances for which the value of $\text{gap}_{\text{UB}_\text{best}}$[{\footnotesize\%}] (the gap between the feasible solution at termination, and the best found feasible solution among the four compared approaches) is
smaller than or equal to the value reported on the horizontal axis.}
\begin{figure}[H]
    \centering
    \includegraphics[width=10cm,height=5.5cm]{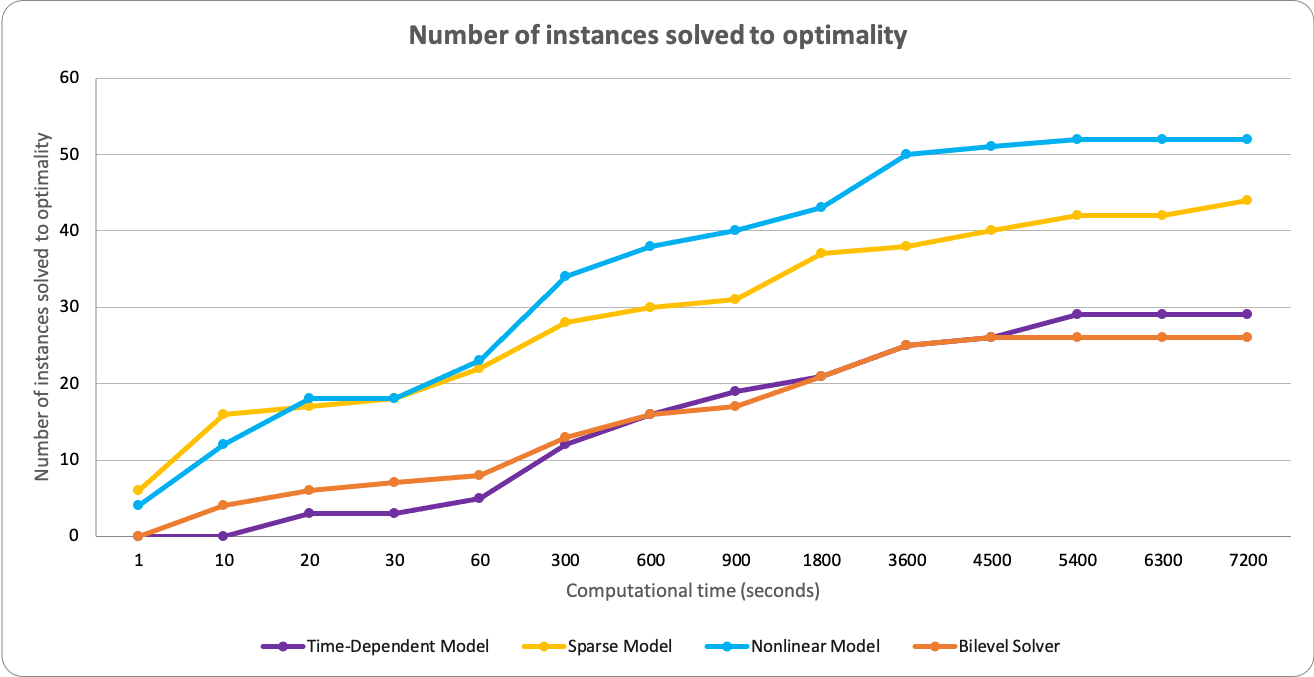}
    \caption{Cumulative chart of number of instances out of 87 solved to optimality within a given computational time.}
    \label{fig:1a}
\end{figure}
\begin{figure}[H]
    \centering
    \subfloat[Cumulative chart of percentage gap at termination.]{\includegraphics[width=7cm,height=5.5cm]{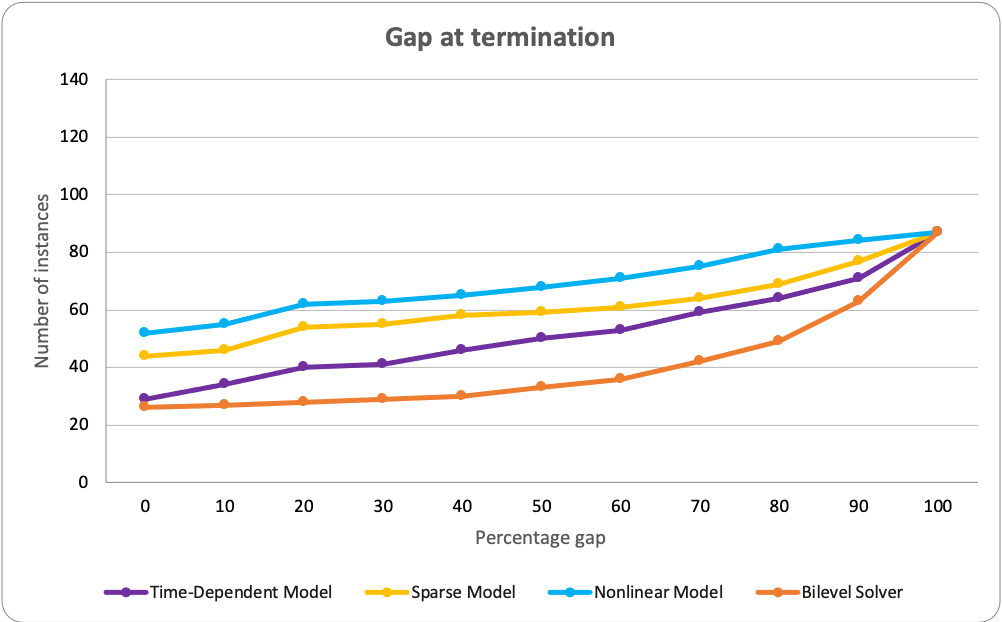}
    \label{fig:2a}}\hfill
    \centering
    \subfloat[Cumulative chart of percentage gap with respect to the best feasible $k$-core at termination.]{\includegraphics[width=7cm,height=5.5cm]{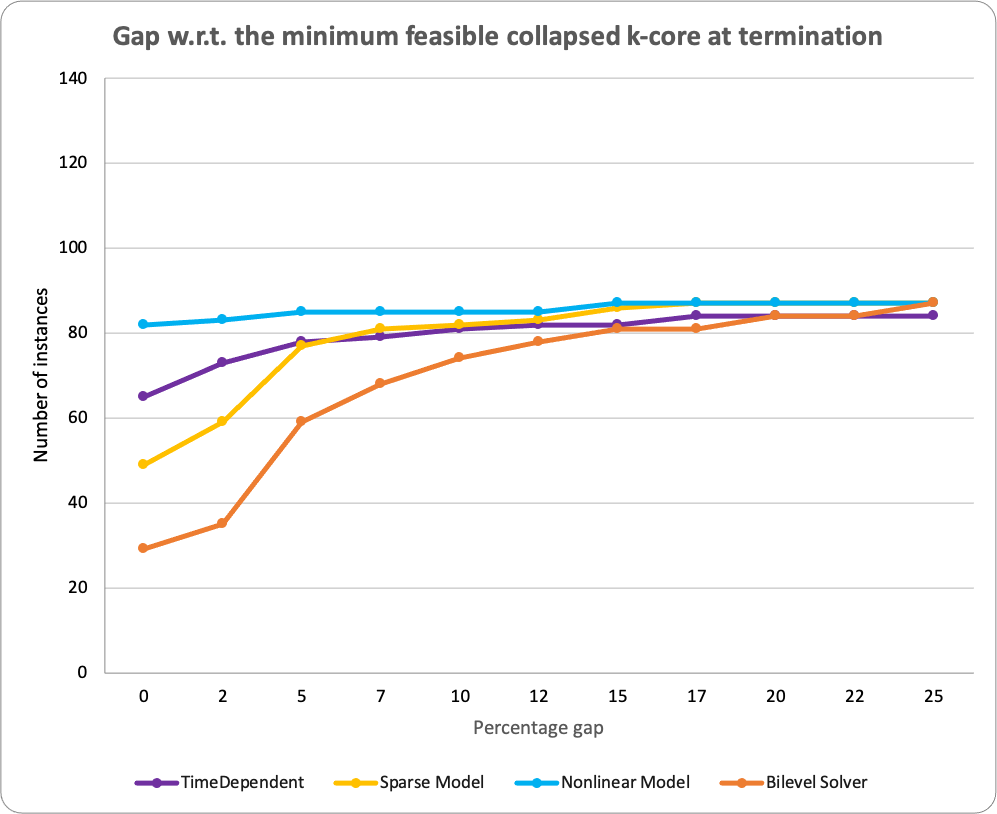}
    \label{fig:3a}}
    \caption{Cumulative charts of two different percentage gaps.\vspace*{-3mm}}
\end{figure}

\end{document}